\documentclass{amsart}
\usepackage[nobysame]{amsrefs}
\usepackage{amssymb}
\usepackage{mathrsfs}
\usepackage[usesnames,svgnames]{xcolor}
\usepackage{tikz,pgf}
	\pgfdeclarelayer{background}
	\pgfsetlayers{background,main}
\usepackage{mathpazo}
\usepackage{todonotes}
\usepackage[figuresright]{rotating}
	\newenvironment{amssidewaysfigure}{\begin{sidewaysfigure}\vspace*{.5\textwidth}\begin{minipage}{\textheight}\centering}{\end{minipage}\end{sidewaysfigure}}\usepackage{enumerate}
\usepackage{subfigure}
\usepackage{etoolbox}
\usepackage{ifthen}
\usepackage[colorlinks=true,
	urlcolor=MidnightBlue,
	citecolor=DarkGreen]{hyperref}
\newtheorem{theorem}{Theorem}[section]
\newtheorem{proposition}[theorem]{Proposition}

\newtheorem{lemma}[theorem]{Lemma}
\newtheorem{corollary}[theorem]{Corollary}

\newtheorem{question}[theorem]{Question}
\theoremstyle{remark}
\newtheorem{remark}[theorem]{Remark}

\newcommand{\ie}{\text{i.e.}\;}
\newcommand{\alert}[1]{{\color{DarkGreen}\emph{#1}}}
\newcommand{\rk}{\text{rk}}
\newcommand{\Bigll}{\Bigl(\!\Bigl(}
\newcommand{\Bigrr}{\Bigr)\!\Bigr)}
\newcommand{\colint}[2]{#1^{(#2)}}
\newcommand{\colref}[3]{\Bigll\colint{#1}{0}\;\colint{#2}{#3}\Bigrr}
\newcommand{\nc}[2]{\ifstrequal{#2}{1}{N\!C_{#1}}{N\!C_{#1}^{(#2)}}}
\newcommand{\pnc}[2]{\ifstrequal{#2}{1}{\mathcal{N\!C}_{#1}}{\mathcal{N\!C}_{#1}^{(#2)}}}
\newcommand{\DD}{\mathcal{D}}
\newcommand{\PP}{\mathcal{P}}
\newcommand{\QQ}{\mathcal{Q}}
\newcommand{\RR}{\mathcal{R}}
\renewcommand{\th}{^{\text{th}}}
\newcommand{\underindex}[2]{\underset{\underset{#1}{\uparrow}}{#2}}
\author{Henri M{\"u}hle}
\address{LIX, {\'E}cole Polytechnique, 91128 Palaiseau, France}
\email{henri.muehle@lix.polytechnique.fr}
\thanks{This work is supported by a Public Grant overseen by the French National Research Agency (ANR) as part of the ``Investissements d'Avenir'' Program (Reference: ANR-10-LABX-0098), and by Digiteo project PAAGT (Nr. 2015-3161D)}
\title[Symmetric Decompositions and the Strong Sperner Property]{Symmetric Decompositions and the Strong Sperner Property for Noncrossing Partition Lattices}

\begin{document}

\begin{abstract}
	We prove that the noncrossing partition lattices associated with the complex reflection groups $G(d,d,n)$ for $d,n\geq 2$ admit symmetric decompositions into Boolean subposets.  As a result, these lattices have the strong Sperner property and their rank-generating polynomials are symmetric, unimodal, and $\gamma$-nonnegative.  We use computer computations to complete the proof that every noncrossing partition lattice associated with a well-generated complex reflection group is strongly Sperner, thus answering affirmatively a question raised by D.~Armstrong.
\end{abstract}

\maketitle

\section{Introduction}
	\label{sec:introduction}
The lattice $\pnc{n}{1}$ of noncrossing partitions of an $n$-element set, introduced by G.~Kreweras \cite{kreweras72sur}, is a remarkable lattice with beautiful enumerative and structural properties.  It is a pure-shellable, (locally) self-dual and complemented lattice, whose cardinality is given by the $n\th$ Catalan number; see \cites{edelman80chain,kreweras72sur} for other interesting enumerative invariants.  Its maximal chains encode parking functions of length $n-1$, and afford a natural transitive action (called the Hurwitz action) of the braid group.  The lattice $\pnc{n}{1}$ has appeared and proved to be significant in a variety of contexts, such as algebraic topology, geometry, free probability, and representation theory; see \cites{mccammond06noncrossing,simion00noncrossing} for surveys on this topic.

Noncrossing partitions can be viewed as elements of the symmetric group.  This connection was perhaps first made explicit by P.~Biane \cite{biane97some} via a map which transforms blocks of a set partition into cycles of a permutation.  An algebraization of this connection allowed for defining analogous posets for every well-generated complex reflection group \cites{bessis03dual,brady01partial,brady02artin,brady02partial,reiner97non}; this definition is recalled in Section~\ref{sec:noncrossing_partitions}.  It was shown that these posets are always lattices, but to date no uniform proof of this fact is available.  A beautiful geometric argument by T.~Brady and C.~Watt establishes this result simultaneously for all real reflection groups~\cite{brady08non}, but for the remaining well-generated complex reflection groups we still have to rely on a case-by-case verification~\cite{bessis06non,bessis15finite}.  Consequently, bearing the prototypical example of $\pnc{n}{1}$ in mind, these posets are called \alert{noncrossing partition lattices} associated with a well-generated complex reflection group $W$; denoted by $\pnc{W}{1}$.  

It is a natural question to ask which properties of $\pnc{n}{1}$ survive the transition to $\pnc{W}{1}$, and whether these properties can be proved without using the classification of irreducible well-generated complex reflection groups or not.  For instance, there is a uniform formula for the cardinality of $\pnc{W}{1}$ in terms of the degrees of $W$, see~\cite{reiner97non}*{Remark~2} and \cite{bessis15finite}*{Section~13}, but no uniform proof is available.  There are formulas for the M{\"o}bius number~\cites{athanasiadis07shellability,armstrong09generalized}, and for the number of maximal chains of $\pnc{W}{1}$~\cite{chapuy14counting,reading08chains}.  The shellability of the order complex of $\pnc{W}{1}$ was established uniformly when $W$ is a real reflection group~\cite{athanasiadis07shellability}, and case-by-case for the remaining groups~\cite{muehle15el}.  Likewise, the transitivity of the Hurwitz action of the braid group on the maximal chains of $\pnc{W}{1}$ was shown uniformly when $W$ is a real reflection group~\cite{deligne74letter}, and case-by-case for the remaining groups~\cite{bessis15finite}.  It seems that the main obstacle for extending the known uniform proofs for real reflection groups to the remaining well-generated complex reflection groups is the absence of a well-behaved root system in the latter groups\footnote{Here, the emphasis is on \emph{well-behaved}; root systems do exist for complex reflection groups.  See for instance~\cite{lehrer09unitary}*{Chapter~2, Section~6}.}.  

The main purpose of this article is to show that yet another property of $\pnc{n}{1}$, namely the \alert{strong Sperner property}, holds in $\pnc{W}{1}$ for every well-generated complex reflection group $W$.  The origin of this property goes back to a classical result of E.~Sperner, who showed in \cite{sperner28ein} that the maximum size of a family of pairwise incomparable subsets of an $n$-element set is $\tbinom{n}{\lfloor\tfrac{n}{2}\rfloor}$.  This result was generalized by P.~Erd{\H o}s, who showed that the maximum size of a $k$-family, \ie a family of subsets of an $n$-element set that does not contain $k+1$ pairwise comparable sets, is the sum of the $k$ largest binomial coefficients~\cite{erdos45on}*{Theorem~5}.  The case $k=1$ clearly yields Sperner's original result.  These notions can be easily rephrased as poset properties as follows.  Given a graded poset $\PP=(P,\leq)$, a $k$-family is a set $X\subseteq P$ that does not contain a chain of length $k+1$.  Then, $\PP$ is \alert{$k$-Sperner} if the size of a $k$-family does not exceed the sum of the $k$ largest rank numbers, and $\PP$ is \alert{strongly Sperner} if it is $k$-Sperner for all $k>0$.  Examples of strongly Sperner posets include the Boolean lattices, the lattices of divisors of an integer~\cite{bruijn51on}, and the Bruhat posets associated with parabolic quotients of finite Coxeter groups other than $H_{4}$~\cite{stanley80weyl}\footnote{This does not mean that the Bruhat poset associated with parabolic quotients of $H_{4}$ does not have the strong Sperner property, though.  The methods applied in \cite{stanley80weyl} simply did not work for $H_{4}$.}.  Some posets that in general lack the Sperner property are geometric lattices~\cite{dilworth71counterexample}, in particular lattices of set partitions of a sufficiently large set~\cites{canfield78on,jichang84superantichains,shearer79simple}.

Despite the fact that the lattice of all set partitions need not be strongly Sperner, the restriction to noncrossing set partitions does have this property.  This result was established by R.~Simion and D.~Ullman in \cite{simion91on}, by showing that the lattice $\pnc{n}{1}$ admits a decomposition into saturated chains that are symmetric about the middle ranks.  (In fact, they proved the stronger result that $\pnc{n}{1}$ can be decomposed into Boolean lattices that are symmetric about the middle ranks.)  Such a \alert{symmetric chain decomposition} (respectively \alert{symmetric Boolean decomposition}) implies the strong Sperner property, while the converse is not necessarily true~\cites{griggs80on,leclerc94finite}.  The key property to constructing a symmetric chain or Boolean decomposition for $\pnc{n}{1}$ is the fact that its intervals are isomorphic to direct products of smaller noncrossing partition lattices, and that the existence of a symmetric chain or Boolean decomposition is preserved under direct products.  Then, R.~Simion and D.~Ullman grouped the elements of $\pnc{n}{1}$ according to the smallest integer that lies in the same block as the integer $1$, and showed that the resulting decomposition is already symmetric, and can be further decomposed into symmetric chains or Boolean lattices.  See \cite{simion91on}*{Section~2} for the details. 

It is therefore reasonable to ask the following questions.

\begin{question}\label{qu:nc_sbd_scd}
	Does the lattice of $W$-noncrossing partitions always admit a symmetric decomposition into saturated chains or Boolean lattices for every well-generated complex reflection group $W$?
\end{question}

\begin{question}[See also \cite{armstrong09generalized}*{Open Problem~3.5.12}]\label{qu:nc_ssp}
	Is the lattice of $W$-noncrossing partitions always strongly Sperner for every well-generated complex reflection group $W$?
\end{question}

An affirmative answer to Question~\ref{qu:nc_sbd_scd} automatically yields an affirmative answer to Question~\ref{qu:nc_ssp}.  Strictly speaking, noncrossing partition lattices are only defined for irreducible well-generated complex reflection groups, but this definition can be extended to all well-generated complex reflection groups by taking the direct product of the noncrossing partition lattices associated with the irreducible factors of the group.  According to \cite{shephard54finite}, the irreducible well-generated complex reflection groups fall into four categories: they are either isomorphic to $G(1,1,n)$ for some $n\geq 1$, to $G(d,1,n)$ for some $d\geq 2$ and some $n\geq 1$, to $G(d,d,n)$ for some $d,n\geq 2$, or they are one of $26$ exceptional groups.  We review this classification in Section~\ref{sec:complex_reflection_groups} in more detail.  For the moment it shall suffice to note that the group $G(1,1,n)$ is isomorphic to the symmetric group of rank $n$, which yields $\pnc{G(1,1,n)}{1}\cong\pnc{n}{1}$.  D.~Bessis and R.~Corran showed in \cite{bessis06non} that $\pnc{G(2,1,n)}{1}\cong\pnc{G(d,1,n)}{1}$ for all $d\geq 2$ and all $n\geq 1$, and V.~Reiner showed in \cite{reiner97non}*{Theorem~13} that the lattice $\pnc{G(2,1,n)}{1}$ admits a symmetric chain decomposition for all $n\geq 1$, a result that was extended by P.~Hersh in \cite{hersh99deformation}*{Theorems~5~and~7} to symmetric Boolean decompositions.  Hence Questions~\ref{qu:nc_sbd_scd} and \ref{qu:nc_ssp} need to be answered for the groups $G(d,d,n)$ and each $d,n\geq 2$, as well as for the exceptional groups.  Our first main result is an affirmative answer for the groups $G(d,d,n)$.

\begin{theorem}\label{thm:gddn_sbd}
	For every $d,n\geq 2$, the lattice $\pnc{G(d,d,n)}{1}$ of $G(d,d,n)$-noncrossing partitions admits a symmetric Boolean decomposition.  Consequently, it admits a symmetric chain decomposition, it is strongly Sperner, and the sequence of its rank numbers is symmetric, unimodal, and $\gamma$-nonnegative.
\end{theorem}

Our proof of Theorem~\ref{thm:gddn_sbd} uses an idea similar to the aforementioned construction of a symmetric chain decomposition of $\pnc{n}{1}$ due to R.~Simion and D.~Ullman.  We view the elements of $\pnc{G(d,d,n)}{1}$ as permutations of a set consisting of $n$ integers each occurring in $d$ different colors, and we group them according to the image of the first integer in the first color.  A careful analysis of this decomposition shows that it produces parts that sit not symmetrically in $\pnc{G(d,d,n)}{1}$.  We overcome this issue by slightly modifying this decomposition, and eventually prove Theorem~\ref{thm:gddn_sbd} by induction.  We remark that according to \cite{athanasiadis04noncrossing}, E.~Tzanaki previously constructed a symmetric chain decomposition of $\pnc{G(2,2,n)}{1}$, which unfortunately did not appear in print.

A consequence of our decomposition of $\pnc{G(d,d,n)}{1}$ is a recursive formula for its rank vector, see Proposition~\ref{prop:gddn_rank_recursion}.  If we had an explicit formula for this vector, we could derive an explicit formula for the corresponding $\gamma$-vector, and since Theorem~\ref{thm:gddn_sbd} asserts that this vector consists of nonnegative integers, it is a natural question to ask for combinatorial interpretations of these numbers.  A combinatorial interpretation for the entries of the $\gamma$-vector of $\pnc{n}{1}$ is for instance given in \cite{petersen15eulerian}*{Section~4.3}.  See also \cite{petersen15eulerian}*{Section~4.7} for more background on $\gamma$-vectors and symmetric Boolean decompositions.

For the exceptional irreducible well-generated complex reflection groups, we approach Questions~\ref{qu:nc_sbd_scd} and \ref{qu:nc_ssp} with the help of a computer.  The main obstacle in answering Question~\ref{qu:nc_sbd_scd} is that it is extremely difficult from a computational perspective; it essentially amounts to checking certain properties for each antichain of the poset.  We managed to complete this computation for almost all exceptional groups of rank at most $4$.  A priori, an answer to Question~\ref{qu:nc_ssp} is equally difficult, however we used a decomposition argument (Proposition~\ref{prop:strongly_sperner_rank_removal}) to simplify the computation.  This enabled us to use \textsc{Sage} \cites{sage,sagecombinat} to provide an affirmative answer to Question~\ref{qu:nc_ssp}.

\begin{theorem}\label{thm:nc_strongly_sperner}
	Let $W$ be a well-generated complex reflection group.  The lattice $\pnc{W}{1}$ of $W$-noncrossing partitions is strongly Sperner, and the sequence of its rank numbers is symmetric, unimodal and $\gamma$-nonnegative.
\end{theorem}

In his thesis~\cite{armstrong09generalized}, D.~Armstrong generalized the noncrossing partition lattice associated with a Coxeter group $W$ by adding a parameter $m$ so that one obtains a certain partial order on the multichains of length $m$ of $\pnc{W}{1}$, and this construction naturally extends to well-generated complex reflection groups.  Let us denote the resulting poset by $\pnc{W}{m}$.  Among many other things, he posed the question whether these \alert{$m$-divisible noncrossing partition posets} are strongly Sperner for any positive integer $m$, see \cite{armstrong09generalized}*{Open Problem~3.5.12}.  Our Theorem~\ref{thm:nc_strongly_sperner} yields an affirmative answer for the case $m=1$, but it remains open for $m\geq 2$.

\smallskip

The rest of the article is organized as follows.  In Section~\ref{sec:preliminaries} we recall the necessary notions, in particular the definitions of strongly Sperner posets and symmetric decompositions (Section~\ref{sec:posets}), complex reflection groups (Section~\ref{sec:complex_reflection_groups}), and noncrossing partitions (Section~\ref{sec:noncrossing_partitions}).  In Section~\ref{sec:decomposition} we prove Theorem~\ref{thm:gddn_sbd}, and the proof of Theorem~\ref{thm:nc_strongly_sperner} is assembled in Section~\ref{sec:remaining_cases}.

\section{Preliminaries}
	\label{sec:preliminaries}
In this section we recall the necessary definitions that we use in this article.  For further background on partially ordered sets we recommend \cite{davey02introduction}, an excellent introduction to the Sperner property and related subjects is \cite{anderson02combinatorics}.  An extensive textbook on complex reflection groups is \cite{lehrer09unitary}, and a recent exposition on Coxeter elements is \cite{reiner14on}.  Throughout the paper we use the abbreviation $[n]=\{1,2,\ldots,n\}$ for an integer $n$, and we consider only finite posets.

\subsection{Partially Ordered Sets}
	\label{sec:posets}
Let $\PP=(P,\leq)$ be a partially ordered set (\alert{poset} for short).  Given two elements $p,q\in P$, we say that $q$ \alert{covers} $p$ if $p<q$ and there exists no $x\in P$ with $p<x<q$.  In this case, we also say that $p$ \alert{is covered by} $q$ or that $p$ and $q$ \alert{form a covering}, and we usually write $p\lessdot q$.  If $\PP$ has a least element, say $\hat{0}$, then every $p\in P$ with $\hat{0}\lessdot p$ is an \alert{atom} of $\PP$.  Dually, if $\PP$ has a greatest element, say $\hat{1}$, then every $p\in P$ with $p\lessdot\hat{1}$ is a \alert{coatom} of $\PP$.  A \alert{closed interval} of $\PP$ is a subset of $P$ that can be written in the form $[p,q]=\{x\in P\mid p\leq x\leq q\}$ for some $p,q\in P$ with $p\leq q$.  A \alert{chain} of $\PP$ is a subset of $P$ that can be written as $C=\{p_{1},p_{2},\ldots,p_{s}\}$ such that $p_{1}<p_{2}<\cdots<p_{s}$, and the length of a chain is its cardinality minus one.  A chain is \alert{saturated} if it is a sequence of coverings.  A saturated chain is \alert{maximal} if it contains a minimal and a maximal element of $\PP$.  A poset is \alert{graded} if all maximal chains have the same length, which we call the \alert{rank} of $\PP$ and denote by $\rk(\PP)$.  We can now define the \alert{rank function} of $\PP$ by
\begin{displaymath}
	\rk:P\to\mathbb{N},\quad x\mapsto\begin{cases}0, & \text{if}\;x\;\text{is a minimal element},\\\rk\bigl([m,x]\bigr), & \text{otherwise, for some minimal element}\;m<x.\end{cases}
\end{displaymath}
A \alert{lattice} is a poset in which any two elements have a least upper bound and a greatest lower bound.  

Given two posets $\PP=(P,\leq_{P})$ and $\QQ=(Q,\leq_{Q})$, their \alert{direct product} is the poset $\PP\times\QQ=(P\times Q,\leq)$, where $(p_{1},q_{1})\leq (p_{2},q_{2})$ if and only if $p_{1}\leq_{P}p_{2}$ and $q_{1}\leq_{Q}q_{2}$.  

%If we denote disjoint set union by $\uplus$, then, by abuse of notation, the \alert{disjoint union} is the poset $\PP\uplus\QQ=(P\uplus Q,\leq)$, where $x\leq y$ if and only if either $x,y\in P$ and $x\leq_{P}y$ or $x,y\in Q$ and $x\leq_{Q}y$.  

A \alert{decomposition} of $\PP=(P,\leq)$ is a partition of $P$ with the property that if $D$ is a part of this partition, then $D$ cannot be written as a disjoint union of two or more nonempty subposets of $\PP$, and each cover relation in $(D,\leq)$ is a cover relation in $\PP$.  A decomposition of a graded poset is \alert{symmetric} if for each part $D$ there is a bijection from the minimal elements of $(D,\leq)$ to the maximal elements of $(D,\leq)$, and if $p$ is a minimal with corresponding maximal element $q$, then $\rk(p)+\rk(q)=\rk(\PP)$.  A \alert{symmetric chain decomposition} is a symmetric decomposition of $\PP$ into chains, and a \alert{symmetric Boolean decomposition} is a symmetric decomposition of $\PP$ into Boolean lattices.  See Figure~\ref{fig:poset_decompositions} for some examples.  Observe that the poset in Figure~\ref{fig:poset_decomposition_3} does not admit a symmetric Boolean decomposition, since it does not have enough elements.  The following observation follows from the fact that each Boolean lattice admits a symmetric chain decomposition.

\begin{proposition}[\cite{petersen13on}*{Observation~10}]\label{prop:sbd_implies_scd}
	A graded poset that admits a symmetric Boolean decomposition also admits a symmetric chain decomposition.
\end{proposition}

\begin{figure}
	\centering
	\subfigure[A decomposition of a graded poset into three parts.]{\label{fig:poset_decomposition_1}
		\begin{tikzpicture}\small
			\def\x{1};
			\def\y{.75};
			\def\s{.67};
			\draw(.5*\x,2*\y) node{};
			\draw(3.5*\x,2*\y) node{};
			\draw(2*\x,1*\y) node[draw,circle,scale=\s](n1){};
			\draw(1*\x,2*\y) node[draw,circle,scale=\s](n2){};
			\draw(2*\x,2*\y) node[draw,circle,scale=\s](n3){};
			\draw(3*\x,2*\y) node[draw,circle,scale=\s](n4){};
			\draw(1*\x,3*\y) node[draw,circle,scale=\s](n5){};
			\draw(1.75*\x,3*\y) node[draw,circle,scale=\s](n6){};
			\draw(2.25*\x,3*\y) node[draw,circle,scale=\s](n7){};
			\draw(3*\x,3*\y) node[draw,circle,scale=\s](n8){};
			\draw(1*\x,4*\y) node[draw,circle,scale=\s](n9){};
			\draw(2*\x,4*\y) node[draw,circle,scale=\s](n10){};
			\draw(3*\x,4*\y) node[draw,circle,scale=\s](n11){};
			\draw(2*\x,5*\y) node[draw,circle,scale=\s](n12){};
			\draw(n1) -- (n2) -- (n6) -- (n11) -- (n12);
			\draw(n1) -- (n4) -- (n7) -- (n9) -- (n12);
			\draw(n1) -- (n3) -- (n6) -- (n10) -- (n12);
			\draw(n2) -- (n5) -- (n9);
			\draw(n3) -- (n7) -- (n10);
			\draw(n4) -- (n8) -- (n11);
			\begin{pgfonlayer}{background}
				\fill[green!50!black,opacity=.5](n1) circle(5pt);
				\fill[green!50!black,opacity=.5](n3) circle(5pt);
				\fill[green!50!black,opacity=.5](n4) circle(5pt);
				\fill[green!50!black,opacity=.5](n6) circle(5pt);
				\fill[green!50!black,opacity=.5](n8) circle(5pt);
				\draw[rounded corners,green!50!black,opacity=.4,cap=round,line width=5pt](n1) -- (n3) -- (n6);
				\draw[rounded corners,green!50!black,opacity=.4,cap=round,line width=5pt](n1) -- (n4) -- (n8);
				\fill[orange!50!black,opacity=.5](n2) circle(5pt);
				\fill[orange!50!black,opacity=.5](n5) circle(5pt);
				\fill[orange!50!black,opacity=.5](n9) circle(5pt);
				\draw[rounded corners,orange!50!black,opacity=.4,cap=round,line width=5pt](n2) -- (n5) -- (n9);
				\fill[red!50!black,opacity=.5](n7) circle(5pt);
				\fill[red!50!black,opacity=.5](n10) circle(5pt);
				\fill[red!50!black,opacity=.5](n11) circle(5pt);
				\fill[red!50!black,opacity=.5](n12) circle(5pt);
				\draw[rounded corners,red!50!black,opacity=.4,cap=round,line width=5pt](n7) -- (n10) -- (n12) -- (n11);
			\end{pgfonlayer}
		\end{tikzpicture}
	}\hspace*{1cm}
	\subfigure[A symmetric decomposition of the poset in Figure~\ref{fig:poset_decomposition_1}.]{\label{fig:poset_decomposition_2}
		\begin{tikzpicture}\small
			\def\x{1};
			\def\y{.75};
			\def\s{.67};
			\draw(.5*\x,2*\y) node{};
			\draw(3.5*\x,2*\y) node{};
			\draw(2*\x,1*\y) node[draw,circle,scale=\s](n1){};
			\draw(1*\x,2*\y) node[draw,circle,scale=\s](n2){};
			\draw(2*\x,2*\y) node[draw,circle,scale=\s](n3){};
			\draw(3*\x,2*\y) node[draw,circle,scale=\s](n4){};
			\draw(1*\x,3*\y) node[draw,circle,scale=\s](n5){};
			\draw(1.75*\x,3*\y) node[draw,circle,scale=\s](n6){};
			\draw(2.25*\x,3*\y) node[draw,circle,scale=\s](n7){};
			\draw(3*\x,3*\y) node[draw,circle,scale=\s](n8){};
			\draw(1*\x,4*\y) node[draw,circle,scale=\s](n9){};
			\draw(2*\x,4*\y) node[draw,circle,scale=\s](n10){};
			\draw(3*\x,4*\y) node[draw,circle,scale=\s](n11){};
			\draw(2*\x,5*\y) node[draw,circle,scale=\s](n12){};
			\draw(n1) -- (n2) -- (n6) -- (n11) -- (n12);
			\draw(n1) -- (n4) -- (n7) -- (n9) -- (n12);
			\draw(n1) -- (n3) -- (n6) -- (n10) -- (n12);
			\draw(n2) -- (n5) -- (n9);
			\draw(n3) -- (n7) -- (n10);
			\draw(n4) -- (n8) -- (n11);
			\begin{pgfonlayer}{background}
				\fill[green!50!black,opacity=.5](n4) circle(5pt);
				\fill[green!50!black,opacity=.5](n5) circle(5pt);
				\fill[green!50!black,opacity=.5](n7) circle(5pt);
				\fill[green!50!black,opacity=.5](n8) circle(5pt);
				\fill[green!50!black,opacity=.5](n9) circle(5pt);
				\draw[rounded corners,green!50!black,opacity=.4,cap=round,line width=5pt](n8) -- (n4) -- (n7) -- (n9) -- (n5);
				\fill[orange!50!black,opacity=.5](n1) circle(5pt);
				\fill[orange!50!black,opacity=.5](n2) circle(5pt);
				\fill[orange!50!black,opacity=.5](n3) circle(5pt);
				\fill[orange!50!black,opacity=.5](n6) circle(5pt);
				\fill[orange!50!black,opacity=.5](n10) circle(5pt);
				\fill[orange!50!black,opacity=.5](n11) circle(5pt);
				\fill[orange!50!black,opacity=.5](n12) circle(5pt);
				\draw[rounded corners,orange!50!black,opacity=.4,cap=round,line width=5pt](n3) -- (n1) -- (n2) -- (n6) -- (n11) -- (n12) -- (n10);
			\end{pgfonlayer}
		\end{tikzpicture}
	}\hspace*{1cm}
	\subfigure[A symmetric chain decomposition of the poset in Figure~\ref{fig:poset_decomposition_1} into four chains, one of which is a singleton sitting on the middle rank.]{\label{fig:poset_decomposition_3}
		\begin{tikzpicture}\small
			\def\x{1};
			\def\y{.75};
			\def\s{.67};
			\draw(.5*\x,2*\y) node{};
			\draw(3.5*\x,2*\y) node{};
			\draw(2*\x,1*\y) node[draw,circle,scale=\s](n1){};
			\draw(1*\x,2*\y) node[draw,circle,scale=\s](n2){};
			\draw(2*\x,2*\y) node[draw,circle,scale=\s](n3){};
			\draw(3*\x,2*\y) node[draw,circle,scale=\s](n4){};
			\draw(1*\x,3*\y) node[draw,circle,scale=\s](n5){};
			\draw(1.75*\x,3*\y) node[draw,circle,scale=\s](n6){};
			\draw(2.25*\x,3*\y) node[draw,circle,scale=\s](n7){};
			\draw(3*\x,3*\y) node[draw,circle,scale=\s](n8){};
			\draw(1*\x,4*\y) node[draw,circle,scale=\s](n9){};
			\draw(2*\x,4*\y) node[draw,circle,scale=\s](n10){};
			\draw(3*\x,4*\y) node[draw,circle,scale=\s](n11){};
			\draw(2*\x,5*\y) node[draw,circle,scale=\s](n12){};
			\draw(n1) -- (n2) -- (n6) -- (n11) -- (n12);
			\draw(n1) -- (n4) -- (n7) -- (n9) -- (n12);
			\draw(n1) -- (n3) -- (n6) -- (n10) -- (n12);
			\draw(n2) -- (n5) -- (n9);
			\draw(n3) -- (n7) -- (n10);
			\draw(n4) -- (n8) -- (n11);
			\begin{pgfonlayer}{background}
				\fill[green!50!black,opacity=.5](n1) circle(5pt);
				\fill[green!50!black,opacity=.5](n2) circle(5pt);
				\fill[green!50!black,opacity=.5](n5) circle(5pt);
				\fill[green!50!black,opacity=.5](n9) circle(5pt);
				\fill[green!50!black,opacity=.5](n12) circle(5pt);
				\draw[rounded corners,green!50!black,opacity=.4,cap=round,line width=5pt](n1) -- (n2) -- (n5) -- (n9) -- (n12);
				\fill[orange!50!black,opacity=.5](n3) circle(5pt);
				\fill[orange!50!black,opacity=.5](n6) circle(5pt);
				\fill[orange!50!black,opacity=.5](n10) circle(5pt);
				\draw[rounded corners,orange!50!black,opacity=.4,cap=round,line width=5pt](n3) -- (n6) -- (n10);
				\fill[blue!50!black,opacity=.5](n7) circle(5pt);
				\fill[red!50!black,opacity=.5](n4) circle(5pt);
				\fill[red!50!black,opacity=.5](n8) circle(5pt);
				\fill[red!50!black,opacity=.5](n11) circle(5pt);
				\draw[rounded corners,red!50!black,opacity=.4,cap=round,line width=5pt](n4) -- (n8) -- (n11);
			\end{pgfonlayer}
		\end{tikzpicture}
	}
	\subfigure[A symmetric Boolean decomposition into four parts.]{\label{fig:poset_decomposition_4}
		\begin{tikzpicture}\small
			\def\x{1};
			\def\y{.75};
			\def\s{.67};
			\draw(6*\x,1*\y) node[draw,circle,scale=\s](n1){};
			\draw(3.5*\x,2*\y) node[draw,circle,scale=\s](n2){};
			\draw(4.5*\x,2*\y) node[draw,circle,scale=\s](n3){};	%a
			\draw(5.5*\x,2*\y) node[draw,circle,scale=\s](n4){};	%b
			\draw(6.5*\x,2*\y) node[draw,circle,scale=\s](n5){};	%c
			\draw(7.5*\x,2*\y) node[draw,circle,scale=\s](n6){};	%d
			\draw(8.5*\x,2*\y) node[draw,circle,scale=\s](n7){};
			\draw(1*\x,3*\y) node[draw,circle,scale=\s](n8){};
			\draw(2*\x,3*\y) node[draw,circle,scale=\s](n9){};
			\draw(3*\x,3*\y) node[draw,circle,scale=\s](n10){};		%ab
			\draw(4*\x,3*\y) node[draw,circle,scale=\s](n11){};		%ac
			\draw(5*\x,3*\y) node[draw,circle,scale=\s](n12){};		%bc
			\draw(6*\x,3*\y) node[draw,circle,scale=\s](n13){};
			\draw(7*\x,3*\y) node[draw,circle,scale=\s](n14){};		%ad
			\draw(8*\x,3*\y) node[draw,circle,scale=\s](n15){};		%bd
			\draw(9*\x,3*\y) node[draw,circle,scale=\s](n16){};		%cd
			\draw(10*\x,3*\y) node[draw,circle,scale=\s](n17){};
			\draw(11*\x,3*\y) node[draw,circle,scale=\s](n18){};
			\draw(3.5*\x,4*\y) node[draw,circle,scale=\s](n19){};
			\draw(4.5*\x,4*\y) node[draw,circle,scale=\s](n20){};	%abc
			\draw(5.5*\x,4*\y) node[draw,circle,scale=\s](n21){};	%abd
			\draw(6.5*\x,4*\y) node[draw,circle,scale=\s](n22){};	%acd
			\draw(7.5*\x,4*\y) node[draw,circle,scale=\s](n23){};	%bcd
			\draw(8.5*\x,4*\y) node[draw,circle,scale=\s](n24){};
			\draw(6*\x,5*\y) node[draw,circle,scale=\s](n25){};
			\draw(n1) -- (n2) -- (n8) -- (n19) -- (n25);
			\draw(n1) -- (n3) -- (n10) -- (n20) -- (n25);
			\draw(n1) -- (n4) -- (n10) -- (n21) -- (n25);
			\draw(n1) -- (n5) -- (n11) -- (n22) -- (n25);
			\draw(n1) -- (n6) -- (n15) -- (n23) -- (n25);
			\draw(n1) -- (n7) -- (n17) -- (n24) -- (n25);
			\draw(n2) -- (n9) -- (n19);
			\draw(n2) -- (n10) -- (n19);
			\draw(n2) -- (n13);
			\draw(n3) -- (n11) -- (n20);
			\draw(n3) -- (n14) -- (n21);
			\draw(n4) -- (n12) -- (n20);
			\draw(n4) -- (n15) -- (n21);
			\draw(n5) -- (n12) -- (n23);
			\draw(n5) -- (n16) -- (n22);
			\draw(n6) -- (n14) -- (n22);
			\draw(n6) -- (n16) -- (n23);
			\draw(n7) -- (n16) -- (n24);
			\draw(n7) -- (n18) -- (n24);
			\draw(n7) -- (n14);
			\draw(n12) -- (n19);
			\draw(n13) -- (n24);
			\begin{pgfonlayer}{background}
				\fill[green!50!black,opacity=.5](n1) circle(5pt);
				\fill[green!50!black,opacity=.5](n3) circle(5pt);
				\fill[green!50!black,opacity=.5](n4) circle(5pt);
				\fill[green!50!black,opacity=.5](n5) circle(5pt);
				\fill[green!50!black,opacity=.5](n6) circle(5pt);
				\fill[green!50!black,opacity=.5](n10) circle(5pt);
				\fill[green!50!black,opacity=.5](n11) circle(5pt);
				\fill[green!50!black,opacity=.5](n12) circle(5pt);
				\fill[green!50!black,opacity=.5](n14) circle(5pt);
				\fill[green!50!black,opacity=.5](n15) circle(5pt);
				\fill[green!50!black,opacity=.5](n16) circle(5pt);
				\fill[green!50!black,opacity=.5](n20) circle(5pt);
				\fill[green!50!black,opacity=.5](n21) circle(5pt);
				\fill[green!50!black,opacity=.5](n22) circle(5pt);
				\fill[green!50!black,opacity=.5](n23) circle(5pt);
				\fill[green!50!black,opacity=.5](n25) circle(5pt);
				\draw[rounded corners,green!50!black,opacity=.4,cap=round,line width=5pt](n1) -- (n3) -- (n10) -- (n20) -- (n25);
				\draw[rounded corners,green!50!black,opacity=.4,cap=round,line width=5pt](n1) -- (n4) -- (n10) -- (n21) -- (n25);
				\draw[rounded corners,green!50!black,opacity=.4,cap=round,line width=5pt](n1) -- (n5) -- (n11) -- (n22) -- (n25);
				\draw[rounded corners,green!50!black,opacity=.4,cap=round,line width=5pt](n1) -- (n6) -- (n15) -- (n23) -- (n25);
				\draw[rounded corners,green!50!black,opacity=.4,cap=round,line width=5pt](n3) -- (n11) -- (n20);
				\draw[rounded corners,green!50!black,opacity=.4,cap=round,line width=5pt](n3) -- (n14) -- (n21);
				\draw[rounded corners,green!50!black,opacity=.4,cap=round,line width=5pt](n4) -- (n12) -- (n20);
				\draw[rounded corners,green!50!black,opacity=.4,cap=round,line width=5pt](n4) -- (n15) -- (n21);
				\draw[rounded corners,green!50!black,opacity=.4,cap=round,line width=5pt](n5) -- (n12) -- (n23);
				\draw[rounded corners,green!50!black,opacity=.4,cap=round,line width=5pt](n5) -- (n16) -- (n22);
				\draw[rounded corners,green!50!black,opacity=.4,cap=round,line width=5pt](n6) -- (n14) -- (n22);
				\draw[rounded corners,green!50!black,opacity=.4,cap=round,line width=5pt](n6) -- (n16) -- (n23);
				\fill[blue!50!black,opacity=.5](n13) circle(5pt);
				\fill[orange!50!black,opacity=.5](n2) circle(5pt);
				\fill[orange!50!black,opacity=.5](n8) circle(5pt);
				\fill[orange!50!black,opacity=.5](n9) circle(5pt);
				\fill[orange!50!black,opacity=.5](n19) circle(5pt);
				\draw[rounded corners,orange!50!black,opacity=.4,cap=round,line width=5pt](n2) -- (n8) -- (n19);
				\draw[rounded corners,orange!50!black,opacity=.4,cap=round,line width=5pt](n2) -- (n9) -- (n19);
				\fill[red!50!black,opacity=.5](n7) circle(5pt);
				\fill[red!50!black,opacity=.5](n17) circle(5pt);
				\fill[red!50!black,opacity=.5](n18) circle(5pt);
				\fill[red!50!black,opacity=.5](n24) circle(5pt);
				\draw[rounded corners,red!50!black,opacity=.4,cap=round,line width=5pt](n7) -- (n17) -- (n24);
				\draw[rounded corners,red!50!black,opacity=.4,cap=round,line width=5pt](n7) -- (n18) -- (n24);
			\end{pgfonlayer}
		\end{tikzpicture}
	}
	\caption{Examples of poset decompositions.}
	\label{fig:poset_decompositions}
\end{figure}
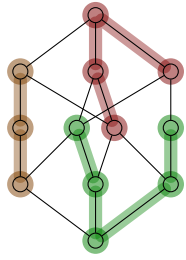
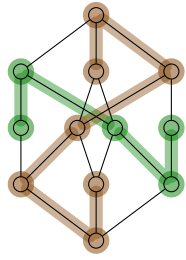
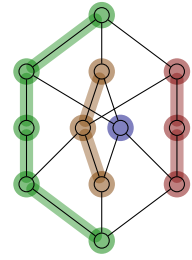
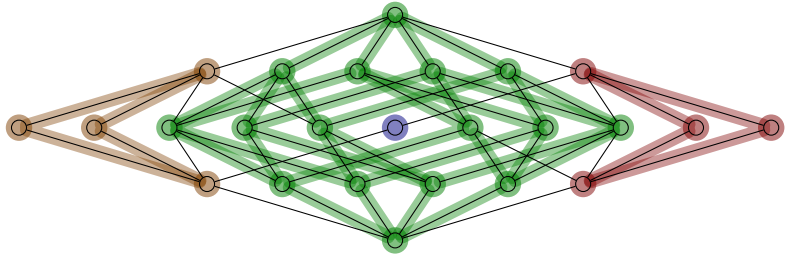

In the remainder of this section we collect a few structural consequences of the existence of a symmetric decomposition into chains or Boolean lattices.  For that let $\PP=(P,\leq)$ be a graded poset of rank $n$, and let
\begin{displaymath}
	\RR_{\PP}(t)=r_{0}(\PP)+r_{1}(\PP)t+\cdots+r_{n}(\PP)t^{n}
\end{displaymath} 
be its \alert{rank-generating polynomial}, \ie the polynomial whose coefficients are defined by $r_{k}(\PP) = \bigl\lvert\{x\in P\mid\rk(x)=k\}\bigr\rvert$.  The sequence of coefficients of $\RR_{\PP}(t)$ is the \alert{rank vector} of $\PP$.  We call $\PP$ \alert{rank-symmetric} if $r_{j}(\PP)=r_{n-j}(\PP)$ for $j\in\{0,1,\ldots,\lfloor\tfrac{n}{2}\rfloor\}$, and we call $\PP$ \alert{rank-unimodal} if there exists some $j\in\{0,1,\ldots,n\}$ such that $r_{0}(\PP)\leq r_{1}(\PP)\leq\cdots\leq r_{j}(\PP)\geq r_{j+1}(\PP)\geq\cdots\geq r_{n}(\PP)$.  If $\PP$ is rank-symmetric, then we can write 
\begin{displaymath}
	\RR_{\PP}(t) = \sum_{j=0}^{\lfloor\frac{n}{2}\rfloor}{\gamma_{j}t^{j}(1+t)^{n-2j}},
\end{displaymath}
and we call the sequence $(\gamma_{0},\gamma_{1},\ldots,\gamma_{\lfloor\tfrac{n}{2}\rfloor})$ the \alert{$\gamma$-vector} of $\PP$.  If the $\gamma$-vector of $\PP$ does not contain negative entries, then $\PP$ is \alert{rank-$\gamma$-nonnegative}.  For example, the poset in Figures~\ref{fig:poset_decomposition_1}--\ref{fig:poset_decomposition_3} has rank-generating polynomial 
\begin{displaymath}
	\RR(t) = 1 + 3t + 4t^2 + 3t^3 + t^4 = (1+t)^4 - t(1+t)^2.
\end{displaymath}
Therefore, its $\gamma$-vector is $(1,-1,0)$.  On the other hand, the poset in Figure~\ref{fig:poset_decomposition_4} is rank-$\gamma$-nonnegative, since its rank-generating polynomial is
\begin{displaymath}
	\RR(t) = 1 + 6t + 11t^2 + 6t^3 + t^4 = (1+t)^4 + 2t(1+t)^2 + t^2, 
\end{displaymath}
which yields the $\gamma$-vector $(1,2,1)$.

A \alert{$k$-family} is a subset of $P$ that does not contain a chain of length $k+1$.  A $1$-family is usually called an \alert{antichain}.  A graded poset is \alert{$k$-Sperner} if the size of a maximal $k$-family equals the sum of the $k$ largest rank numbers, and it is \alert{strongly Sperner} if it is $k$-Sperner for all $k\in[n]$.  See Figure~\ref{fig:sperner_posets} for some examples.  The next result states the connection between symmetric chain decompositions and the strong Sperner property.

\begin{proposition}[\cite{engel97sperner}*{Lemma~5.1.1~and~Theorem~5.1.4}]\label{prop:symmetric_chains_strongly_sperner}
	If a graded poset admits a symmetric chain decomposition, then it is strongly Sperner, rank-symmetric, and rank-unimodal. 
\end{proposition}

The existence of a symmetric Boolean decomposition has an even stronger consequence. 

\begin{proposition}[\cite{petersen13on}*{Observation~11}]\label{prop:symmetric_boolean_gamma_nonnegative}
	If a graded poset of rank $n$ admits a symmetric Boolean decomposition, then it is rank-$\gamma$-nonnegative.  In fact, the $j\th$ entry in the $\gamma$-vector equals the number of parts in this decomposition with cardinality $2^{n-2j}$.  
\end{proposition}

\begin{figure}
	\centering
	\subfigure[A strongly Sperner poset.]{\label{fig:sperner_posets_1}
		\begin{tikzpicture}\small
			\def\x{1};
			\def\y{.75};
			\def\s{.67};
			\draw(.5*\x,2*\y) node{};
			\draw(3.5*\x,2*\y) node{};
			\draw(2*\x,1*\y) node[draw,circle,scale=\s](n1){};
			\draw(1*\x,2*\y) node[draw,circle,scale=\s](n2){};
			\draw(2*\x,2*\y) node[draw,circle,scale=\s](n3){};
			\draw(3*\x,2*\y) node[draw,circle,scale=\s](n4){};
			\draw(1*\x,3*\y) node[draw,circle,scale=\s](n5){};
			\draw(1.75*\x,3*\y) node[draw,circle,scale=\s](n6){};
			\draw(2.25*\x,3*\y) node[draw,circle,scale=\s](n7){};
			\draw(3*\x,3*\y) node[draw,circle,scale=\s](n8){};
			\draw(1*\x,4*\y) node[draw,circle,scale=\s](n9){};
			\draw(2*\x,4*\y) node[draw,circle,scale=\s](n10){};
			\draw(3*\x,4*\y) node[draw,circle,scale=\s](n11){};
			\draw(2*\x,5*\y) node[draw,circle,scale=\s](n12){};
			\draw(n1) -- (n2) -- (n6) -- (n11) -- (n12);
			\draw(n1) -- (n4) -- (n7) -- (n9) -- (n12);
			\draw(n1) -- (n3) -- (n6) -- (n10) -- (n12);
			\draw(n2) -- (n5) -- (n9);
			\draw(n3) -- (n7) -- (n10);
			\draw(n4) -- (n8) -- (n11);
		\end{tikzpicture}
	}\hspace*{1cm}
	\subfigure[A Sperner poset that is not $2$-Sperner.]{\label{fig:sperner_posets_2}
		\begin{tikzpicture}\small
			\def\x{1};
			\def\y{.75};
			\def\s{.67};
			\draw(.5*\x,2*\y) node{};
			\draw(3.5*\x,2*\y) node{};
			\draw(2*\x,1*\y) node[draw,circle,scale=\s](n1){};
			\draw(1*\x,2*\y) node[draw,circle,scale=\s](n2){};
			\draw(2*\x,2*\y) node[draw,circle,scale=\s](n3){};
			\draw(3*\x,2*\y) node[draw,circle,scale=\s](n4){};
			\draw(1*\x,3*\y) node[draw,circle,scale=\s](n5){};
			\draw(1.75*\x,3*\y) node[draw,circle,scale=\s](n6){};
			\draw(2.25*\x,3*\y) node[draw,circle,scale=\s](n7){};
			\draw(3*\x,3*\y) node[draw,circle,scale=\s](n8){};
			\draw(1*\x,4*\y) node[draw,circle,scale=\s](n9){};
			\draw(2*\x,4*\y) node[draw,circle,scale=\s](n10){};
			\draw(3*\x,4*\y) node[draw,circle,scale=\s](n11){};
			\draw(2*\x,5*\y) node[draw,circle,scale=\s](n12){};
			\draw(n1) -- (n2) -- (n5) -- (n9) -- (n12);
			\draw(n1) -- (n3) -- (n6) -- (n9);
			\draw(n1) -- (n4) -- (n7) -- (n10) -- (n12);
			\draw(n4) -- (n8) -- (n11) -- (n12);
		\end{tikzpicture}
	}\hspace*{1cm}
	\subfigure[A $2$-Sperner poset that is not Sperner.]{\label{fig:sperner_posets_3}
		\begin{tikzpicture}\small
			\def\x{1};
			\def\y{.75};
			\def\s{.67};
			\draw(.5*\x,2*\y) node{};
			\draw(3.5*\x,2*\y) node{};
			\draw(2*\x,1*\y) node[draw,circle,scale=\s](n1){};
			\draw(1*\x,2*\y) node[draw,circle,scale=\s](n2){};
			\draw(2*\x,2*\y) node[draw,circle,scale=\s](n3){};
			\draw(3*\x,2*\y) node[draw,circle,scale=\s](n4){};
			\draw(1*\x,4*\y) node[draw,circle,scale=\s](n5){};
			\draw(2*\x,4*\y) node[draw,circle,scale=\s](n6){};
			\draw(3*\x,4*\y) node[draw,circle,scale=\s](n7){};
			\draw(2*\x,5*\y) node[draw,circle,scale=\s](n8){};
			\draw(n1) -- (n2) -- (n5) -- (n8);
			\draw(n1) -- (n3) -- (n5);
			\draw(n1) -- (n4) -- (n6) -- (n8);
			\draw(n4) -- (n7) -- (n8);
		\end{tikzpicture}
	}
	\caption{Some illustrations on the Sperner property.}
	\label{fig:sperner_posets}
\end{figure}
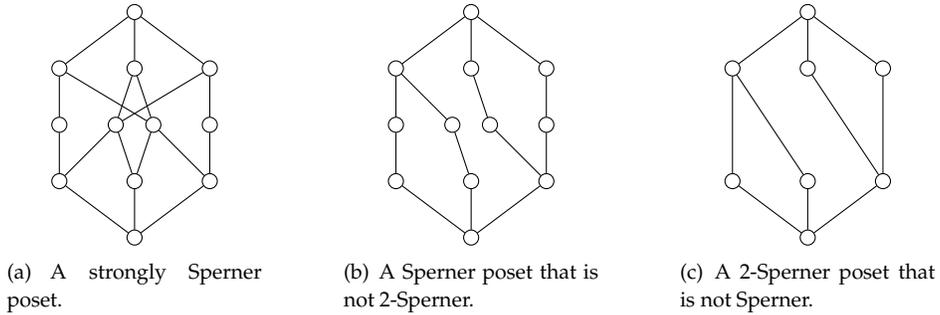

\subsection{Complex Reflection Groups}
	\label{sec:complex_reflection_groups}
A \alert{reflection} is a unitary transformation $t$ on an $n$-dimensional complex vector space $V$ that has finite order and fixes a subspace of $V$ of codimension $1$, the so-called \alert{reflecting hyperplane} associated with $t$.  A subgroup $W$ of the group of all unitary transformations on $V$ is a \alert{complex reflection group} if it is generated by reflections.  If $W$ does not preserve a proper subspace of $V$, then $W$ is \alert{irreducible}, and the \alert{rank} of $W$ is the codimension of the space fixed by $W$.  If $W$ has rank $n$ and can be generated by $n$ reflections, then we call $W$ \alert{well-generated}.  Any maximal subgroup of $W$ that fixes a subspace of $V$ pointwise is a \alert{parabolic subgroup} of $W$.

An important property that distinguishes complex reflection groups from other finite groups is that its algebra of invariant polynomials is again a polynomial algebra~\cites{chevalley55invariants,shephard54finite}.  Moreover, if we choose the generators of this algebra homogeneously, then the corresponding degrees become group invariants, and will simply be called the \alert{degrees} of $W$.  We usually denote the degrees of $W$ by $d_{1},d_{2},\ldots,d_{n}$, where we implicitly assume that their values increase weakly.

According to the classification of irreducible complex reflection groups due to G.~C.~Shephard and J.~A.~Todd~\cite{shephard54finite}, there is one infinite family of such groups, parametrized by three integers $d,e,n$, whose members are usually denoted by $G(de,e,n)$, as well as $34$ exceptional groups, usually denoted by $G_{4},G_{5},\ldots,G_{37}$.  The groups $G(de,e,n)$ can be realized as groups of monomial $n\times n$ matrices, \ie matrices with a unique non-zero entry in each row and in each column.  For a monomial matrix to belong to $G(de,e,n)$ its non-zero entries need to be $(de)\th$ roots of unity, while the product of all its non-zero entries needs to be a $d\th$ root of unity.  Consequently these groups possess a wreath product structure,  see~\cite{lehrer09unitary}*{Chapter~2.2} for the details.  It follows from \cite{orlik80unitary}*{Table~2} that there are three infinite families of irreducible well-generated complex reflection groups, namely $G(1,1,n)$ for some $n\geq 1$, $G(d,1,n)$ for some $n\geq 1$ and some $d\geq 2$, and $G(d,d,n)$ for some $d,n\geq 2$, as well as $26$ exceptional irreducible well-generated complex reflection groups.

We remark that the finite irreducible Coxeter groups are the irreducible well-generated complex reflection groups realizable over a real vector space; we have the following correspondences:
\begin{itemize}
	\item the group $G(1,1,n)$ is isomorphic to the Coxeter group $A_{n-1}$ (which in turn is isomorphic to the symmetric group of rank $n$),
	\item the group $G(2,1,n)$ is isomorphic to the Coxeter group $B_{n}$ (which in turn is isomorphic to the hyperoctahedral group of rank $n$),
	\item the group $G(2,2,n)$ is isomorphic to the Coxeter group $D_{n}$,
	\item the group $G(d,d,2)$ is isomorphic to the Coxeter group $I_{2}(d)$ (which in turn is isomorphic to the dihedral group of order $2d$), \quad and
	\item the groups $G_{23},G_{28},G_{30},G_{35},G_{36},G_{37}$ are isomorphic to the Coxeter groups $H_{3},F_{4},H_{4},E_{6},E_{7},E_{8}$, respectively.	
\end{itemize}

\subsection{Coxeter Elements}
	\label{sec:coxeter_elements}
A vector $\mathbf{v}\in V$ is \alert{regular} if it does not lie in any of the reflecting hyperplanes of $W$.  If $\zeta$ is an eigenvalue of $w\in W$, and the corresponding eigenspace contains a regular vector, then we say that $w$ is \alert{$\zeta$-regular}.  The multiplicative order $d$ of $\zeta$ is a \alert{regular number} for $W$.  The $\zeta$-regular elements of $W$ form a single conjugacy class~\cite{springer74regular}*{Theorem~4.2}.  

If $W$ is irreducible and well-generated, then there exist some well-behaved regular elements.  More precisely, in that case it follows from \cite{lehrer99reflection}*{Theorem~C} that the largest degree is always a regular number for $W$, and we usually write $h$ instead of $d_{n}$, and call it the \alert{Coxeter number} of $W$.  A \alert{Coxeter element} is a $\zeta$-regular element of order $h$ for some $h\th$ root of unity $\zeta$~\cite{reiner14on}*{Definition~1.1}.

\subsection{Noncrossing Partitions}
	\label{sec:noncrossing_partitions}
Let $W$ be an irreducible well-generated complex reflection group, and let $T\subseteq W$ denote the set of all reflections of $W$.  The \alert{absolute length} of $w\in W$ is defined by
\begin{equation}\label{eq:absolute_length}
	\ell_{T}(w) = \min\{k\mid w=t_{1}t_{2}\cdots t_{k}\;\text{for}\;t_{i}\in T\}.
\end{equation}
The \alert{absolute order} on $W$ is the partial order $\leq_{T}$ defined by
\begin{equation}\label{eq:absolute_order}
	u\leq_{T}v\quad\text{if and only if}\quad\ell_{T}(v)=\ell_{T}(u)+\ell_{T}(u^{-1}v),
\end{equation}
for all $u,v\in W$.  For every Coxeter element $\gamma\in W$, define the set of \alert{$W$-noncrossing partitions} by
\begin{equation}\label{eq:noncrossing_partitions}
	\nc{W}{1}(\gamma) = \{u\in W\mid \varepsilon\leq_{T}u\leq_{T}\gamma\},
\end{equation}
where $\varepsilon$ denotes the identity of $W$.  The poset $\pnc{W}{1}(\gamma)=\bigl(\nc{W}{1}(\gamma),\leq_{T}\bigr)$ is the \alert{lattice of $W$-noncrossing partitions}, and its structure does not depend on the choice of $\gamma$ as the next result shows.  We thus suppress the Coxeter element in the notation whenever it is not necessary.  

\begin{proposition}[\cite{reiner14on}*{Corollary~1.6}]\label{prop:noncrossing_partitions_independent}
	Let $W$ be an irreducible well-generated complex reflection group.  The posets $\pnc{W}{1}(\gamma)$ and $\pnc{W}{1}(\gamma')$ are isomorphic for all Coxeter elements $\gamma,\gamma'\in W$.	
\end{proposition}

The fact that $\pnc{W}{1}$ is a lattice was shown by a collaborative effort of several authors \cites{athanasiadis04noncrossing,bessis03dual,bessis06non,bessis15finite,brady01partial,brady02artin,kreweras72sur,reiner97non}.  To date a uniform proof of the lattice property of $\pnc{W}{1}$ is only available for the real reflection groups~\cite{brady08non}.

%In \cite{brady01partial} the lattice property was shown for $\pnc{G(1,1,n)}{1}$, where it was also shown that this lattice coincides with the lattice of noncrossing set partitions of $[n]$ introduced by G.~Kreweras in \cite{kreweras72sur}.  (Hence the name!)  The same observation was made by P.~Biane in \cite{biane97some}, however, without explicitly using the algebraic definition in \eqref{eq:absolute_order}.  In \cite{brady02artin} it was shown that $\pnc{G(2,1,n)}{1}$ and $\pnc{(G,2,2,n)}{1}$ are lattices, and it was observed that $\pnc{G(2,1,n)}{1}$ coincides with the type-$B$ lattice of noncrossing set partitions introduced by V.~Reiner in \cite{reiner97non}.  A combinatorial model for $\pnc{G(2,2,n)}{1}$ was later given by C.~Athanasiadis and V.~Reiner in \cite{athanasiadis04noncrossing}.  D.~Bessis used a different perspective to show that $\pnc{W}{1}$ is a lattice whenever $W$ is a real reflection group (\ie whenever $W$ can be realized as a finite reflection group acting on a real vector space)~\cite{bessis03dual}.  Together with R.~Corran he gave a combinatorial model for $\pnc{G(d,d,n)}{1}$ and proved that it is a lattice~\cite{bessis06non}.  Finally, he proved the lattice property for all remaining well-generated complex reflection groups in \cite{bessis15finite} in a case-by-case fashion.  To date a uniform proof of the lattice property of $\pnc{W}{1}$ is only available for the real reflection groups~\cite{brady08non}.

The noncrossing partition lattices enjoy many nice structural properties.  It is straightforward from the definition that they are graded, atomic, (locally) self-dual, and (locally) complemented, and it is a little more involved to show that they are also EL-shellable~\cites{athanasiadis07shellability,bjorner80shellable,muehle15el,reiner97non}.  Another striking property is that the cardinality of $\pnc{W}{1}$ is given by the corresponding $W$-Catalan number, defined by
\begin{equation}\label{eq:coxeter_catalan_number}
	\text{Cat}_{W} = \prod_{i=1}^{n}{\frac{d_{i}+h}{d_{i}}}.
\end{equation}
This was observed for the real reflection groups in \cites{athanasiadis04noncrossing,chapoton04enumerative,reiner97non}, and in the general case in \cites{bessis06non,bessis15finite}.  Again, to date a uniform proof of this fact is not available.

For later use, we record the following observation

\begin{proposition}[\cite{ripoll10orbites}*{Proposition~6.3(i),(ii)}]\label{prop:parabolic_coxeter_elements}
	Let $W$ be an irreducible well-generated complex reflection group, and let $w\in W$.  Let $T$ denote the set of all reflections of $W$.  The following are equivalent:
	\begin{enumerate}[(i)]
		\item $w$ is a Coxeter element in a parabolic subgroup of $W$, and
		\item there is a Coxeter element $\gamma_{w}\in W$ such that $w\leq_{T}\gamma_{w}$.
	\end{enumerate}
\end{proposition}

A \alert{parabolic Coxeter element} is any $w\in W$ that satisfies one of the properties stated in Proposition~\ref{prop:parabolic_coxeter_elements}.  

\begin{remark}
	We need to be a little careful when using Proposition~\ref{prop:parabolic_coxeter_elements} together with our definition of Coxeter elements from Section~\ref{sec:coxeter_elements}.  The reference \cite{ripoll10orbites} uses a less general definition of a Coxeter element (it considers only those coming from the regular number $e^{2i\pi/h}$).  However, the results in \cite{reiner14on} guarantee that Proposition~\ref{prop:parabolic_coxeter_elements} also holds in the more general setting.	
%	statement of Proposition~\ref{prop:parabolic_coxeter_elements} remains true in the more general setting.  Proposition~1.5 in \cite{reiner14on} states that any two Coxeter elements of an irreducible well-generated complex reflection group are related by a reflection automorphism, \ie a group automorphism that fixes the set of all reflections.  It is immediate from the definition that a reflection automorphism $\psi$ of $W$ that sends a Coxeter element $\gamma$ to a Coxeter element $\gamma'$ constitutes an isomorphism of lattices from $\pnc{W}{1}(\gamma)$ to $\pnc{W}{1}(\gamma')$.  Let us therefore assume that $\gamma$ is a Coxeter element according to the definition used in \cite{ripoll10orbites}.  If we pick some $w\leq_{T}\gamma$ and denote by $W'$ the parabolic subgroup of $W$ in which $w$ is a Coxeter element, then it follows from the fact that $\psi$ is a reflection automorphism that $\psi(W')$ is a parabolic subgroup of $W$ in which $\psi(w)$ is a Coxeter element and vice versa.
\end{remark}

\section{Decompositions of $\pnc{G(d,d,n)}{1}$}
	\label{sec:decomposition}
\subsection{The Setup}
	\label{sec:setup}
In this section, we focus on the irreducible well-generated complex reflection group $G(d,d,n)$ for some fixed choice of $d,n\geq 2$.  Recall from Section~\ref{sec:complex_reflection_groups} that the elements of $G(d,d,n)$ can be realized as monomial $n\times n$ matrices whose non-zero entries are $d\th$ roots of unity and for which the product of all non-zero entries is $1$.  In this representation, we can view $G(d,d,n)$ as a subgroup of the symmetric group $\mathfrak{S}_{dn}$ acting on the set
\begin{displaymath}
	\Bigl\{\colint{1}{0},\colint{2}{0},\ldots,\colint{n}{0},\colint{1}{1},\colint{2}{1},\ldots,\colint{n}{1},\ldots,\colint{1}{d-1},\colint{2}{d-1},\ldots,\colint{n}{d-1}\Bigr\}
\end{displaymath}
of $n$ integers with $d$ colors, where $\colint{k}{s}$ represents the column vector whose $k\th$ entry is $\zeta^{s}$ for some primitive $d\th$ root of unity $\zeta$, and whose other entries are zero.  In particular, $w\in G(d,d,n)$ satisfies
\begin{displaymath}
	w\Bigl(\colint{k}{s}\Bigr) = \colint{\pi(k)}{s+t_{k}}\quad\text{and}\quad\sum_{i=1}^{k}{t_{k}}\equiv 0\pmod{d},
\end{displaymath}
where $\pi\in\mathfrak{S}_{n}$ and the numbers $t_{k}$ depend only on $w$ and $k$.  (Here, addition in the superscript is considered modulo $d$.)  More precisely, the permutation $\pi$ is given by the permutation matrix that is derived from $w$ by replacing each non-zero entry by $1$, and the number $t_{k}$ is determined by the non-zero value in position $(k,\pi(k))$ of $w$.  Consequently, we can decompose the elements of $G(d,d,n)$ into generalized cycles of the following form
\begin{multline*}
	\Bigll\colint{k_{1}}{t_{1}}\;\colint{k_{2}}{t_{2}}\;\ldots\;\colint{k_{r}}{t_{r}}\Bigrr = \Bigl(\colint{k_{1}}{t_{1}}\;\colint{k_{2}}{t_{2}}\;\ldots\;\colint{k_{r}}{t_{r}}\Bigr)\Bigl(\colint{k_{1}}{t_{1}+1}\;\colint{k_{2}}{t_{2}+1}\;\ldots\;\colint{k_{r}}{t_{r}+1}\Bigr)\\
		\cdots\Bigl(\colint{k_{1}}{t_{1}+d-1}\;\colint{k_{2}}{t_{2}+d-1}\;\ldots\;\colint{k_{r}}{t_{r}+d-1}\Bigr),
\end{multline*}
and
\begin{multline*}
	\Bigl[\colint{k_{1}}{t_{1}}\;\colint{k_{2}}{t_{2}}\;\ldots\;\colint{k_{r}}{t_{r}}\Bigr]_{s} = \Bigl(\colint{k_{1}}{t_{1}}\;\colint{k_{2}}{t_{2}}\;\ldots\;\colint{k_{r}}{t_{r}}\;\colint{k_{1}}{t_{1}+s}\;\colint{k_{2}}{t_{2}+s}\\
		\ldots\;\colint{k_{r}}{t_{r}+s}\;\ldots\;\colint{k_{1}}{t_{1}(d-1)s}\;\colint{k_{2}}{t_{2}+(d-1)s}\;\ldots\;\colint{k_{r}}{t_{r}+(d-1)s}\Bigr),
\end{multline*}
for $s\in[d-1]$.  We call the first type a \alert{simultaneous cycle} and the second type a \alert{balanced cycle}, and we usually suppress the subscript $1$.  In each case we say that $r$ is the \alert{length} of such a generalized cycle.  Now we recall some general statements.  

\begin{lemma}[\cite{broue98complex}*{Table~5}]\label{lem:gddn_parabolic_subgroups}
	Let $W$ be a parabolic subgroup of $G(d,d,n)$.  If $W$ is irreducible, then it is either isomorphic to $G(1,1,n')$ or to $G(d,d,n')$ for $n'\leq n$.  If $W$ is reducible, then it is isomorphic to a direct product of irreducible parabolic subgroups of $G(d,d,n)$.  
\end{lemma}

\begin{lemma}\label{lem:gddn_single_short_cycles}
	Let $w$ be a parabolic Coxeter element of $G(d,d,n)$.  If $w$ consists of a single simultaneous cycle of length $k$, then the interval $[\varepsilon,w]$ in $\bigl(G(d,d,n),\leq_{T}\bigr)$ is isomorphic to $\pnc{G(1,1,k)}{1}$.  
\end{lemma}
\begin{proof}
	By definition, if $w$ consists of a single simultaneous cycle of length $k$, then the isomorphism type of the interval $[\varepsilon,w]$ in $\bigl(G(d,d,n),\leq_{T}\bigr)$ is uniquely determined by the underlying permutation, regardless of the colors involved.  This follows from the fact that the reflections in $G(d,d,n)$ are of the form $\colref{a}{b}{s}$ for $1\leq a<b\leq n$ and $0\leq s<d$.  In particular, whenever we change the color of some integer, we have to permute it with another integer (and change its color accordingly).  Now Lemma~\ref{lem:gddn_parabolic_subgroups} implies that the parabolic subgroup of $G(d,d,n)$, in which $w$ is a Coxeter element is isomorphic to $G(1,1,k)$, which concludes the proof.
\end{proof}

In what follows we consider the Coxeter element $\gamma\in G(d,d,n)$ represented by the monomial matrix
\begin{equation}\label{eq:gddn_coxeter_matrix}
	\left(\begin{array}{ccccccc}0 & 0 & 0 & \cdots & 0 & \zeta_{d} & 0\\ 1 & 0 & 0 & \cdots & 0 & 0 & 0\\ 0 & 1 & 0 & \cdots & 0 & 0 & 0\\ \vdots & \vdots & \vdots & & \vdots & \vdots & \vdots\\ 0 & 0 & 0 & \cdots & 1 & 0 & 0\\ 0 & 0 & 0 & \cdots & 0 & 0 & \zeta_{d}^{d-1}\end{array}\right),
\end{equation}
where $\zeta_{d}=e^{2i\pi/d}$ is a primitive $d\th$ root of unity.  See for instance \cite{muehle15el}*{Section~3.3} for a justification that this is indeed a Coxeter element of $G(d,d,n)$.  Moreover, $\gamma$ can be expressed as a product of two balanced cycles, namely
\begin{equation}\label{eq:gddn_coxeter_element}
	\gamma = \Bigl[\colint{1}{0}\;\colint{2}{0}\;\ldots\;\colint{(n-1)}{0}\Bigr]\Bigl[\colint{n}{0}\Bigr]_{d-1}.
\end{equation}
The next two lemmas describe the atoms and coatoms in $\pnc{G(d,d,n)}{1}(\gamma)$.  Let $T_{\gamma}=T\cap\nc{W}{1}(\gamma)$.  

\begin{lemma}[\cite{muehle15el}*{Proposition~3.6}]\label{lem:gddn_atoms}
	Let $\gamma$ be the Coxeter element of $G(d,d,n)$ as defined in \eqref{eq:gddn_coxeter_element}.  Then we have
	\begin{multline*}
		T_{\gamma} = \Bigl\{\colref{a}{b}{s}\mid 1\leq a<b<n, s\in\{0,d-1\}\Bigr\}\;\cup\\
			\Bigl\{\colref{a}{n}{s}\mid 1\leq a<n, 0\leq s<d\Bigr\}.
	\end{multline*}
\end{lemma}

\begin{lemma}[\cite{muehle15el}*{Lemma~3.13}]\label{lem:gddn_coatoms}
	Let $\gamma$ be the Coxeter element of $G(d,d,n)$ as defined in \eqref{eq:gddn_coxeter_element}, and let $t\in T_{\gamma}$.  If $t=\colref{a}{b}{0}$ for $1\leq a<b<n$, then we have
	\begin{displaymath}
		\gamma t = \Bigl[\colint{1}{0}\;\ldots\;\colint{a}{0}\;\colint{(b+1)}{0}\;\ldots\;\colint{(n-1)}{0}\Bigr]\Bigl[\colint{n}{0}\Bigr]_{d-1}\Bigll\colint{(a+1)}{0}\;\ldots\;\colint{b}{0}\Bigrr.
	\end{displaymath}
	If $t=\colref{a}{b}{d-1}$ for $1\leq a<b<n$, then we have
	\begin{displaymath}
		\gamma t = \Bigll\colint{1}{0}\;\ldots\;\colint{a}{0}\;\colint{(b+1)}{d-1}\;\ldots\;\colint{(n-1)}{d-1}\Bigrr\Bigl[\colint{(a+1)}{0}\;\ldots\;\colint{b}{0}\Bigr]\Bigl[\colint{n}{0}\Bigr]_{d-1}.
	\end{displaymath}
	If $t=\colref{a}{n}{s}$ for $1\leq a<n$ and $0\leq s<d$, then we have
	\begin{displaymath}
		\gamma t = \Bigll\colint{1}{0}\;\ldots\;\colint{a}{0}\;\colint{n}{s-1}\;\colint{(a+1)}{d-1}\;\ldots\;\colint{(n-1)}{d-1}\Bigrr.
	\end{displaymath}
\end{lemma}

The following lemma describes a translation principle that simplifies computations in noncrossing partition lattices.

\begin{lemma}\label{lem:bottom_intervals}
	Let $[x,y]$ be a non-singleton interval in $\pnc{G(d,d,n)}{1}$.  If $u\leq_{T}x$, then the map 
	\begin{displaymath}	
		f_{u}:[x,y]\to[u^{-1}x,u^{-1}y], \quad w\mapsto u^{-1}w
	\end{displaymath}
	is a poset isomorphism.
\end{lemma}
\begin{proof}
	Let $x\leq a\leq b\leq y$.  Since $u\leq_{T}x$, we obtain
	\begin{align*}
		a\leq_{T}b & \quad\Longleftrightarrow\quad \ell_{T}(b)=\ell_{T}(a)+\ell_{T}(a^{-1}b)\\
		& \quad\Longleftrightarrow\quad \ell_{T}(u)+\ell_{T}(u^{-1}b) = \ell_{T}(u)+\ell_{T}(u^{-1}a)+\ell_{T}(a^{-1}b)\\
		& \quad\Longleftrightarrow\quad \ell_{T}(u^{-1}b) = \ell_{T}(u^{-1}a)+\ell_{T}\bigl((u^{-1}a)^{-1}u^{-1}b\bigr)\\
		& \quad\Longleftrightarrow\quad u^{-1}a\leq_{T}u^{-1}b.
	\end{align*}
\end{proof}

\subsection{A Symmetric Decomposition of $\pnc{G(1,1,n)}{1}$}
	\label{sec:noncrossing_decomposition}
In order to prove that $\pnc{G(d,d,n)}{1}$ admits a symmetric Boolean decomposition for $d,n\geq 2$, we generalize the core idea of \cite{simion91on}*{Theorem~2} which states that $\pnc{G(1,1,n)}{1}(c)$ admits a symmetric Boolean decomposition, where $c=(1\;2\;\ldots\;n)$ is a long cycle in $G(1,1,n)\cong\mathfrak{S}_{n}$.  For $i\in[n]$ define $R_{i}=\bigl\{u\in\nc{G(1,1,n)}{1}(c)\mid u(1)=i\bigr\}$, and let $\RR_{i}=(R_{i},\leq_{T})$.  Moreover, let $\mathbf{2}$ denote the $2$-element chain, let $\uplus$ denote disjoint set union, and let $\RR_{1}\uplus\RR_{2}$ denote the subposet of $\pnc{G(1,1,n)}{1}(c)$ induced by the set $R_{1}\uplus R_{2}$.  We have the following result. 

\begin{theorem}[\cite{simion91on}*{Theorem~2}]\label{thm:g11n_decomposition}
	For $n>0$ we have $\nc{G(1,1,n)}{1}(c)=\biguplus_{i=1}^{n}{R_{i}}$.  Moreover, we have $\RR_{1}\uplus\RR_{2}\cong\mathbf{2}\times\pnc{G(1,1,n-1)}{1}$ and $\RR_{i}\cong\pnc{G(1,1,i-2)}{1}\times\pnc{G(1,1,n-i+1)}{1}$ for $3\leq i\leq n$.  This decomposition is symmetric.
\end{theorem}

Theorem~\ref{thm:g11n_decomposition} has the following immediate corollary.

\begin{corollary}\label{cor:g11n_sbd}
	For $n>0$ the lattice $\pnc{G(1,1,n)}{1}$ admits a symmetric Boolean decomposition, and hence a symmetric chain decomposition.
\end{corollary}
\begin{proof}
	This follows by induction on $n$ from Theorem~\ref{thm:g11n_decomposition} and the fact that symmetric Boolean decompositions are preserved under direct products~\cite{petersen15eulerian}*{Observation~4.4}, as well as Propositions~\ref{prop:sbd_implies_scd} and \ref{prop:noncrossing_partitions_independent}.
\end{proof}

\subsection{A First Decomposition of $\pnc{G(d,d,n)}{1}(\gamma)$}
	\label{sec:first_decomposition}
In the spirit of Theorem~\ref{thm:g11n_decomposition} we define 
\begin{displaymath}
	R_{i}^{(s)} = \Bigl\{u\in\nc{G(d,d,n)}{1}(\gamma)\mid u\bigl(\colint{1}{0}\bigr)=\colint{i}{s}\Bigr\}.
\end{displaymath}
for $i\in[n]$ and $s\in\{0,1,\ldots,d-1\}$.  It is immediately clear that
\begin{displaymath}
	\nc{G(d,d,n)}{1}(\gamma) = \biguplus_{i=1}^{n}\biguplus_{s=0}^{d-1}{R_{i}^{(s)}}.
\end{displaymath}

\begin{lemma}\label{lem:empty_chunks}
	The set $R_{i}^{(s)}$ is empty if and only if either $i=1$ and $2\leq s<d$, or $2\leq i<n$ and $1\leq s<d-1$.	
\end{lemma}
\begin{proof}
	See Appendix~\ref{app:fixed_space}.
\end{proof}

A consequence of Lemma~\ref{lem:empty_chunks} is that we can write
\begin{equation}\label{eq:gddn_first_decomposition}
	\nc{G(d,d,n)}{1}(\gamma) = R_{1}^{(0)}\uplus R_{1}^{(1)}\uplus\biguplus_{i=2}^{n-1}\Bigl(R_{i}^{(0)}\uplus R_{i}^{(d-1)}\Bigr)\uplus\biguplus_{s=0}^{d-1}{R_{n}^{(s)}}.
\end{equation}
In what follows, we describe the isomorphism type of the subposets of $\pnc{G(d,d,n)}{1}(\gamma)$ induced by the nonempty sets $R_{i}^{(s)}$.  Let us abbreviate $\RR_{i}^{(s)}=\bigl(R_{i}^{(s)},\leq_{T}\bigr)$. 

\begin{lemma}\label{lem:first_chunk}
	For $n>2$ the poset $\RR_{1}^{(0)}\uplus\RR_{2}^{(0)}$ is isomorphic to $\mathbf{2}\times\pnc{G(d,d,n-1)}{1}$.  Moreover, its least element has length $0$ and its greatest element has length $n$. 
\end{lemma}
\begin{proof}
	The poset $\RR_{2}^{(0)}$ clearly constitutes a closed interval of $\pnc{G(d,d,n)}{1}(\gamma)$, since $\colref{1}{2}{0}$ and $\gamma$ are both contained in $R_{2}^{(0)}$.  It then follows from Lemma~\ref{lem:bottom_intervals} that the map $x\mapsto\colref{1}{2}{0}x$ is a poset isomorphism from $R_{2}^{(0)}$ to $R_{1}^{(0)}$.  Since the greatest element of $\RR_{1}^{(0)}$ is $\Bigl[\colint{2}{0}\;\cdots\;\colint{(n\!-\!1)}{0}\Bigr]\Bigl[\colint{n}{0}\Bigr]_{d-1}$, we get $\RR_{1}^{(0)}\cong\pnc{G(d,d,n-1)}{1}$.  Since $x\lessdot_{T}\colref{1}{2}{0}x$ for every $x\in R_{1}^{(0)}$, the claim follows.
\end{proof}

\begin{lemma}\label{lem:second_chunk}
	The poset $\RR_{n}^{(s)}$ is isomorphic to $\pnc{G(1,1,n-1)}{1}$ for $0\leq s<d$.  Morever, its least element has length $1$ and its greatest element has length $n-1$.  
\end{lemma}
\begin{proof}
	Fix $s\in\{0,1,\ldots,d-1\}$, and consider the set $R_{n}^{(s)}$.  Lemma~\ref{lem:gddn_atoms} implies that $\colref{1}{n}{s}\in R_{n}^{(s)}$, and Lemma~\ref{lem:gddn_coatoms} implies that the only coatom that maps $\colint{1}{0}$ to $\colint{n}{s}$ is 
	\begin{displaymath}
		\Bigll\colint{1}{0}\;\colint{n}{s}\;\colint{2}{d-1}\;\ldots\;\colint{(n-1)}{d-1}\Bigrr.
	\end{displaymath}
	It follows that $\RR_{n}^{(s)}$ forms a closed interval in $\pnc{G(d,d,n)}{1}(\gamma)$ whose least element has length $1$ and whose greatest element has length $n-1$.  Finally Lemmas~\ref{lem:gddn_single_short_cycles} and \ref{lem:bottom_intervals} imply that $\RR_{n}^{(s)}\cong\pnc{G(1,1,n-1)}{1}$ as desired.
\end{proof}

\begin{lemma}\label{lem:third_chunk}
	The poset $\RR_{i}^{(0)}$ is isomorphic to $\pnc{G(d,d,n-i+1)}{1}\times\pnc{G(1,1,i-2)}{1}$ whenever $3\leq i<n$.  Moreover, its least element has length $1$ and its greatest element has length $n-1$.  
\end{lemma}
\begin{proof}
	 Lemma~\ref{lem:gddn_atoms} implies that $\colref{1}{i}{0}\in R_{i}^{(0)}$, and Lemma~\ref{lem:gddn_coatoms} implies that the only coatom that maps $\colint{1}{0}$ to $\colint{i}{0}$ is
	\begin{equation}\label{eq:third_chunk_coatom}
		\Bigl[\colint{1}{0}\;\colint{i}{0}\;\ldots\;\colint{(n-1)}{0}\Bigr]\Bigl[\colint{n}{0}\Bigr]_{d-1}\Bigll\colint{2}{0}\;\ldots\;\colint{(i-1)}{0}\Bigrr.
	\end{equation}
	It follows once more that $\RR_{i}^{(0)}$ forms a closed interval in $\pnc{G(d,d,n)}{1}(\gamma)$ whose least element has length $1$ and whose greatest element has length $n-1$.  It is also immediate to check that the greatest element of $\RR_{i}^{(0)}$, namely the one in \eqref{eq:third_chunk_coatom}, is a Coxeter element in a parabolic subgroup of $G(d,d,n)$ which is isomorphic to $G(d,d,n-i+2)\times G(1,1,i-2)$.  Hence Lemma~\ref{lem:bottom_intervals} implies that $\RR_{i}^{(0)}$ is isomorphic to $\pnc{G(d,d,n-i+1)}{1}\times\pnc{G(1,1,i-2)}{1}$ as desired.  (We have to reduce the rank of the first factor by one, since the least element in $\RR_{i}^{(0)}$ is not the identity, but rather $\colref{1}{i}{0}$ instead.  Moreover, observe that the rank of $\pnc{G(1,1,i-2)}{1}$ is actually $i-3$.)
\end{proof}

\begin{lemma}\label{lem:fourth_chunk}
	The poset $\RR_{i}^{(d-1)}$ is isomorphic to $\pnc{G(1,1,n-i)}{1}\times\pnc{G(d,d,i-1)}{1}$ whenever $3\leq i<n$.  Moreover, its least element has length $1$ and its greatest element has length $n-1$.  
\end{lemma}
\begin{proof}
	This is essentially the same proof as that of Lemma~\ref{lem:third_chunk}.  Observe that the greatest element of $\RR_{i}^{(d-1)}$ is
	\begin{displaymath}
		\Bigll\colint{1}{0}\;\colint{i}{d-1}\;\ldots\;\colint{(n-1)}{d-1}\Bigrr\Bigl[\colint{2}{0}\;\ldots\;\colint{(i-1)}{0}\Bigr]\Bigl[\colint{n}{0}\Bigr]_{d-1},
	\end{displaymath}
	which is a Coxeter element in a parabolic subgroup of $G(d,d,n)$ isomorphic to $G(1,1,n-i+1)\times G(d,d,i-1)$.
\end{proof}

\begin{corollary}
	For $3\leq i<n$ and $d\geq 2$ we have $\RR_{i}^{(0)}\cong\RR_{n-i+2}^{(d-1)}$.
\end{corollary}
\begin{proof}
	This follows immediately from Lemmas~\ref{lem:third_chunk} and \ref{lem:fourth_chunk}. 
\end{proof}

\begin{lemma}\label{lem:fifth_chunk}
	The posets $\RR_{1}^{(1)}$ and $\RR_{2}^{(d-1)}$ are both isomorphic to $\pnc{G(1,1,n-2)}{1}$.  In the first case the least element has length $2$ and the greatest element has length $n-1$, while in the second case the least element has length $1$ and the greatest element has length $n-2$.  
\end{lemma}
\begin{proof}
	See Appendix~\ref{app:fixed_space}.
\end{proof}

We observe that the induced subposet $\RR_{1}^{(1)}\uplus\RR_{2}^{(d-1)}$ of $\pnc{G(d,d,n)}{1}(\gamma)$ is disconnected whenever $n>2$, see for instance Figure~\ref{fig:g553_nc}.  (In this example $\RR_{1}^{(1)}$ consists only of the coatom $\Bigl[\colint{1}{0}\Bigr]\Bigl[\colint{3}{0}\Bigr]_{4}$, and $\RR_{2}^{(4)}$ consists only of the atom $\colref{1}{2}{4}$, and these two elements are incomparable.)  Furthermore, since both $R_{1}^{(1)}$ and $R_{2}^{(d-1)}$ do not sit in $\pnc{G(d,d,n)}{1}(\gamma)$ symmetrically, the decomposition in \eqref{eq:gddn_first_decomposition} is not a suitable starting point to obtain a symmetric decomposition of $\pnc{G(d,d,n)}{1}(\gamma)$.  We overcome this issue in the next section.  

\begin{amssidewaysfigure}
	\centering
	\begin{tikzpicture}\small
		\def\x{1.75};
		\def\y{5};
		\def\s{.67};
		\draw(6.5*\x,.5*\y) node[scale=\s](n1){$\Bigll\colint{1}{0}\Bigrr$};
		\draw(1*\x,1*\y) node[scale=\s](n2){$\Bigll\colint{2}{0}\;\colint{3}{0}\Bigrr$};
		\draw(2*\x,1*\y) node[scale=\s](n3){$\Bigll\colint{1}{0}\;\colint{3}{0}\Bigrr$};
		\draw(3*\x,1*\y) node[scale=\s](n4){$\Bigll\colint{1}{0}\;\colint{2}{4}\Bigrr$};
		\draw(4*\x,1*\y) node[scale=\s](n5){$\Bigll\colint{1}{0}\;\colint{3}{4}\Bigrr$};
		\draw(5*\x,1*\y) node[scale=\s](n6){$\Bigll\colint{2}{0}\;\colint{3}{1}\Bigrr$};
		\draw(6*\x,1*\y) node[scale=\s](n7){$\Bigll\colint{2}{0}\;\colint{3}{4}\Bigrr$};
		\draw(7*\x,1*\y) node[scale=\s](n8){$\Bigll\colint{1}{0}\;\colint{3}{1}\Bigrr$};
		\draw(8*\x,1*\y) node[scale=\s](n9){$\Bigll\colint{1}{0}\;\colint{3}{3}\Bigrr$};
		\draw(9*\x,1*\y) node[scale=\s](n10){$\Bigll\colint{2}{0}\;\colint{3}{2}\Bigrr$};
		\draw(10*\x,1*\y) node[scale=\s](n11){$\Bigll\colint{1}{0}\;\colint{2}{0}\Bigrr$};
		\draw(11*\x,1*\y) node[scale=\s](n12){$\Bigll\colint{2}{0}\;\colint{3}{3}\Bigrr$};
		\draw(12*\x,1*\y) node[scale=\s](n13){$\Bigll\colint{1}{0}\;\colint{3}{2}\Bigrr$};
		\draw(1*\x,2*\y) node[scale=\s](n14){$\Bigll\colint{1}{0}\;\colint{2}{0}\;\colint{3}{0}\Bigrr$};
		\draw(2*\x,2*\y) node[scale=\s](n15){$\Bigll\colint{1}{0}\;\colint{3}{4}\;\colint{2}{4}\Bigrr$};
		\draw(3*\x,2*\y) node[scale=\s](n16){$\Bigl[\colint{1}{0}\Bigr]\Bigl[\colint{3}{0}\Bigr]_{4}$};
		\draw(4*\x,2*\y) node[scale=\s](n17){$\Bigll\colint{1}{0}\;\colint{3}{0}\;\colint{2}{4}\Bigrr$};
		\draw(5*\x,2*\y) node[scale=\s](n18){$\Bigll\colint{1}{0}\;\colint{2}{0}\;\colint{3}{4}\Bigrr$};
		\draw(6*\x,2*\y) node[scale=\s](n19){$\Bigll\colint{1}{0}\;\colint{2}{0}\;\colint{3}{1}\Bigrr$};
		\draw(7*\x,2*\y) node[scale=\s](n20){$\Bigll\colint{1}{0}\;\colint{3}{3}\;\colint{2}{4}\Bigrr$};
		\draw(8*\x,2*\y) node[scale=\s](n21){$\Bigll\colint{1}{0}\;\colint{3}{1}\;\colint{2}{4}\Bigrr$};
		\draw(9*\x,2*\y) node[scale=\s](n22){$\Bigll\colint{1}{0}\;\colint{2}{0}\;\colint{3}{3}\Bigrr$};
		\draw(10*\x,2*\y) node[scale=\s](n23){$\Bigl[\colint{2}{0}\Bigr]\Bigl[\colint{3}{0}\Bigr]_{4}$};
		\draw(11*\x,2*\y) node[scale=\s](n24){$\Bigll\colint{1}{0}\;\colint{2}{0}\;\colint{3}{2}\Bigrr$};
		\draw(12*\x,2*\y) node[scale=\s](n25){$\Bigll\colint{1}{0}\;\colint{3}{2}\;\colint{2}{4}\Bigrr$};
		\draw(6.5*\x,2.5*\y) node[scale=\s](n26){$\Bigl[\colint{1}{0}\;\colint{2}{0}\Bigr]\Bigl[\colint{3}{0}\bigr]_{4}$};	
		\draw(n1) -- (n2);
		\draw(n1) -- (n3);
		\draw(n1) -- (n4);
		\draw(n1) -- (n5);
		\draw(n1) -- (n6);
		\draw(n1) -- (n7);
		\draw(n1) -- (n8);
		\draw(n1) -- (n9);
		\draw(n1) -- (n10);
		\draw(n1) -- (n11);
		\draw(n1) -- (n12);
		\draw(n1) -- (n13);
		\draw(n2) -- (n14);
		\draw(n2) -- (n15);
		\draw(n2) -- (n23);
		\draw(n3) -- (n14);
		\draw(n3) -- (n16);
		\draw(n3) -- (n17);
		\draw(n4) -- (n15);
		\draw(n4) -- (n17);
		\draw(n4) -- (n20);
		\draw(n4) -- (n21);
		\draw(n4) -- (n25);
		\draw(n5) -- (n15);
		\draw(n5) -- (n16);
		\draw(n5) -- (n18);
		\draw(n6) -- (n17);
		\draw(n6) -- (n19);
		\draw(n6) -- (n23);
		\draw(n7) -- (n18);
		\draw(n7) -- (n20);
		\draw(n7) -- (n23);
		\draw(n8) -- (n16);
		\draw(n8) -- (n19);
		\draw(n8) -- (n21);
		\draw(n9) -- (n16);
		\draw(n9) -- (n20);
		\draw(n9) -- (n22);
		\draw(n10) -- (n21);
		\draw(n10) -- (n23);
		\draw(n10) -- (n24);
		\draw(n11) -- (n14);
		\draw(n11) -- (n18);
		\draw(n11) -- (n19);
		\draw(n11) -- (n22);
		\draw(n11) -- (n24);
		\draw(n12) -- (n22);
		\draw(n12) -- (n23);
		\draw(n12) -- (n25);
		\draw(n13) -- (n16);
		\draw(n13) -- (n24);
		\draw(n13) -- (n25);
		\draw(n14) -- (n26);
		\draw(n15) -- (n26);
		\draw(n16) -- (n26);
		\draw(n17) -- (n26);
		\draw(n18) -- (n26);
		\draw(n19) -- (n26);
		\draw(n20) -- (n26);
		\draw(n21) -- (n26);
		\draw(n22) -- (n26);
		\draw(n23) -- (n26);
		\draw(n24) -- (n26);
		\draw(n25) -- (n26);
		\begin{pgfonlayer}{background}
			\fill[green!50!black,opacity=.5](n1) ellipse(21pt and 10pt);
			\fill[green!50!black,opacity=.5](n2) ellipse(21pt and 10pt);
			\fill[green!50!black,opacity=.5](n6) ellipse(21pt and 10pt);
			\fill[green!50!black,opacity=.5](n7) ellipse(21pt and 10pt);
			\fill[green!50!black,opacity=.5](n10) ellipse(21pt and 10pt);
			\fill[green!50!black,opacity=.5](n11) ellipse(21pt and 10pt);
			\fill[green!50!black,opacity=.5](n12) ellipse(21pt and 10pt);
			\fill[green!50!black,opacity=.5](n14) ellipse(21pt and 10pt);
			\fill[green!50!black,opacity=.5](n18) ellipse(21pt and 10pt);
			\fill[green!50!black,opacity=.5](n19) ellipse(21pt and 10pt);
			\fill[green!50!black,opacity=.5](n22) ellipse(21pt and 10pt);
			\fill[green!50!black,opacity=.5](n23) ellipse(21pt and 10pt);
			\fill[green!50!black,opacity=.5](n24) ellipse(21pt and 10pt);
			\fill[green!50!black,opacity=.5](n26) ellipse(21pt and 10pt);
			\draw[rounded corners,green!50!black,opacity=.4,cap=round,line width=5pt](n1) -- (n2) -- (n23);
			\draw[rounded corners,green!50!black,opacity=.4,cap=round,line width=5pt](n1) -- (n6) -- (n23);
			\draw[rounded corners,green!50!black,opacity=.4,cap=round,line width=5pt](n1) -- (n7) -- (n23);
			\draw[rounded corners,green!50!black,opacity=.4,cap=round,line width=5pt](n1) -- (n10) -- (n23);
			\draw[rounded corners,green!50!black,opacity=.4,cap=round,line width=5pt](n1) -- (n12) -- (n23);
			\draw[rounded corners,green!50!black,opacity=.4,cap=round,line width=5pt](n11) -- (n14) -- (n26);
			\draw[rounded corners,green!50!black,opacity=.4,cap=round,line width=5pt](n11) -- (n18) -- (n26);
			\draw[rounded corners,green!50!black,opacity=.4,cap=round,line width=5pt](n11) -- (n19) -- (n26);
			\draw[rounded corners,green!50!black,opacity=.4,cap=round,line width=5pt](n11) -- (n22) -- (n26);
			\draw[rounded corners,green!50!black,opacity=.4,cap=round,line width=5pt](n11) -- (n24) -- (n26);
			\draw[rounded corners,green!50!black,opacity=.4,cap=round,line width=5pt](n1) -- (n11);
			\draw[rounded corners,green!50!black,opacity=.4,cap=round,line width=5pt](n2) -- (n14);
			\draw[rounded corners,green!50!black,opacity=.4,cap=round,line width=5pt](n6) -- (n19);
			\draw[rounded corners,green!50!black,opacity=.4,cap=round,line width=5pt](n7) -- (n18);
			\draw[rounded corners,green!50!black,opacity=.4,cap=round,line width=5pt](n10) -- (n24);
			\draw[rounded corners,green!50!black,opacity=.4,cap=round,line width=5pt](n12) -- (n22);
			\draw[rounded corners,green!50!black,opacity=.4,cap=round,line width=5pt](n23) -- (n26);
			\fill[blue!50!black,opacity=.5](n3) ellipse(21pt and 10pt);
			\fill[blue!50!black,opacity=.5](n17) ellipse(21pt and 10pt);
			\draw[rounded corners,blue!50!black,opacity=.4,cap=round,line width=5pt](n3) -- (n17);
			\fill[purple!50!black,opacity=.5](n5) ellipse(21pt and 10pt);
			\fill[purple!50!black,opacity=.5](n15) ellipse(21pt and 10pt);
			\draw[rounded corners,purple!50!black,opacity=.4,cap=round,line width=5pt](n5) -- (n15);
			\fill[yellow!50!black,opacity=.5](n8) ellipse(21pt and 10pt);
			\fill[yellow!50!black,opacity=.5](n21) ellipse(21pt and 10pt);
			\draw[rounded corners,yellow!50!black,opacity=.4,cap=round,line width=5pt](n8) -- (n21);
			\fill[red!50!black,opacity=.5](n9) ellipse(21pt and 10pt);
			\fill[red!50!black,opacity=.5](n20) ellipse(21pt and 10pt);
			\draw[rounded corners,red!50!black,opacity=.4,cap=round,line width=5pt](n9) -- (n20);
			\fill[orange!50!black,opacity=.5](n13) ellipse(21pt and 10pt);
			\fill[orange!50!black,opacity=.5](n25) ellipse(21pt and 10pt);
			\draw[rounded corners,orange!50!black,opacity=.4,cap=round,line width=5pt](n13) -- (n25);
			\fill[white!50!black,opacity=.5](n4) ellipse(21pt and 10pt);
			\fill[brown!50!black,opacity=.5](n16) ellipse(21pt and 10pt);			
		\end{pgfonlayer}
	\end{tikzpicture}
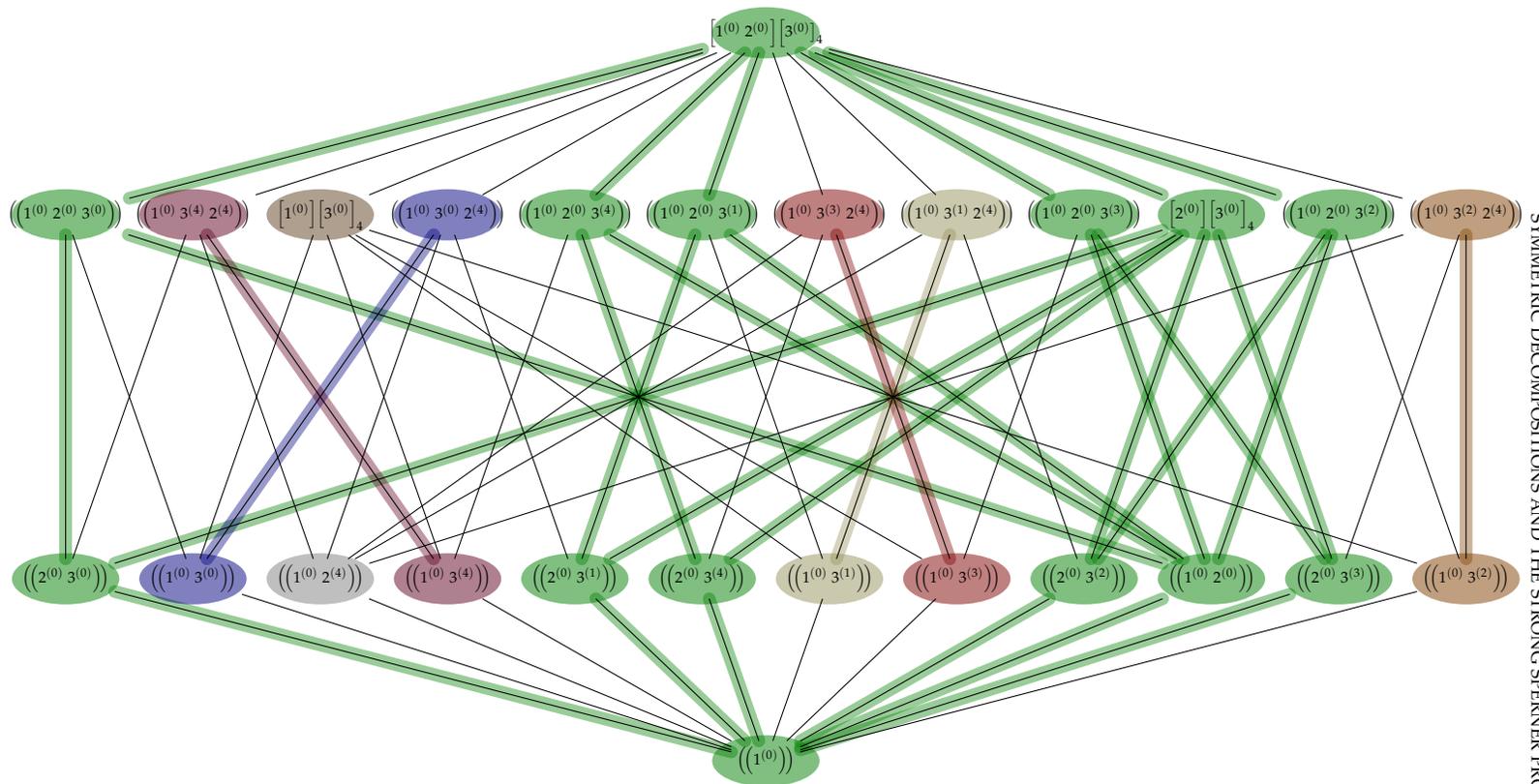
	\caption{The lattice $\pnc{G(5,5,3)}{1}\Bigl(\Bigl[\colint{1}{0}\;\colint{2}{0}\Bigr]\Bigl[\colint{3}{0}\Bigr]_{4}\Bigr)$.  The decomposition according to \eqref{eq:gddn_first_decomposition} is highlighted.}
	\label{fig:g553_nc}
\end{amssidewaysfigure}

\begin{remark}
	We remark that the structure of the decomposition of $\pnc{G(d,d,n)}{1}(\gamma)$ into the pieces $R_{i}^{(s)}$ depends on the choice of Coxeter element.  If we for instance consider the Coxeter element 
	\begin{displaymath}
		\gamma' = \Bigl[\colint{1}{0}\Bigr]_{d-1}\Bigl[\colint{n}{0}\;\colint{(n-1)}{0}\;\ldots\;\colint{2}{0}\Bigr]
	\end{displaymath}
	used by D.~Bessis and R.~Corran in \cite{bessis06non}, then we obtain a decomposition of $\pnc{G(d,d,n)}{1}(\gamma')$ with the following properties:
	\begin{itemize}
		\item the parts $R_{1}^{(s)}$ are empty for $2\leq s<d-1$;
		\item the subposet $\RR_{1}^{(0)}\uplus\RR_{1}^{(d-1)}$ is isomorphic to $\pnc{G(2,1,n-1)}{1}$;\quad and
		\item the subposets $\RR_{i}^{(s)}$ for $2\leq i\leq n$ and $0\leq s<d$ are isomorphic to $\pnc{G(1,1,n-1)}{1}$.
	\end{itemize}
	However, \cite{bessis06non}*{Lemma~1.23(i)--(iii)} implies that the subposet $\RR_{1}^{(0)}\uplus\RR_{1}^{(d-1)}$ is not connected in our sense, since cover relations in $\RR_{1}^{(0)}\uplus\RR_{1}^{(d-1)}$ do not necessarily correspond to cover relations in $\pnc{G(d,d,n)}{1}(\gamma')$.  It is therefore not suitable as a starting point for a symmetric decomposition of $\pnc{G(d,d,n)}{1}(\gamma')$ either.  See Figure~\ref{fig:g552_bc} for an example.  
\end{remark}

\begin{figure}
	\centering
	\begin{tikzpicture}\small
		\def\x{2};
		\def\y{1.5};
		\draw(3*\x,1*\y) node(n1){$\Bigll\colint{1}{0}\Bigrr$};
		\draw(1*\x,2*\y) node(n2){$\colref{1}{2}{0}$};
		\draw(2*\x,2*\y) node(n3){$\colref{1}{2}{1}$};
		\draw(3*\x,2*\y) node(n4){$\colref{1}{2}{2}$};
		\draw(4*\x,2*\y) node(n5){$\colref{1}{2}{3}$};
		\draw(5*\x,2*\y) node(n6){$\colref{1}{2}{4}$};
		\draw(3*\x,3*\y) node(n7){$\Bigl[\colint{1}{0}\Bigr]_{4}\Bigl[\colint{2}{0}\Bigr]$};
		\draw(n1) -- (n2);
		\draw(n1) -- (n3);
		\draw(n1) -- (n4);
		\draw(n1) -- (n5);
		\draw(n1) -- (n6);
		\draw(n2) -- (n7);
		\draw(n3) -- (n7);
		\draw(n4) -- (n7);
		\draw(n5) -- (n7);
		\draw(n6) -- (n7);
	\end{tikzpicture}
	\caption{The lattice $\pnc{G(5,5,2)}{1}\Bigl(\Bigl[\colint{1}{0}\Bigr]_{4}\Bigl[\colint{2}{0}\Bigr]\Bigr)$.  The subposet $\RR_{1}^{(0)}\uplus\RR_{1}^{(4)}$---which consists of the least and the greatest element---is isomorphic to a $2$-chain, but does not form a cover relation in the whole lattice.}
	\label{fig:g552_bc}
\end{figure}

The decomposition \eqref{eq:gddn_first_decomposition} may not be symmetric, but it can be used to derive a recursive formula for the rank-generating polynomial of $\pnc{G(d,d,n)}{1}$.  

\begin{proposition}\label{prop:gddn_rank_recursion}
	Let $d,n\geq 2$, and let $\RR_{\pnc{G(d,d,n)}{1}}(t)=r_{0}(d,n)+r_{1}(d,n)t+\cdots+r_{n}(d,n)t^{n}$ be the rank-generating polynomial of $\pnc{G(d,d,n)}{1}$.  We have $r_{0}(d,n)=1=r_{n}(d,n)$, and $r_{1}(d,n)=(n-1)(n+d-2)=r_{n-1}(d,n)$.  For $k\in\{2,3,\ldots,n-2\}$ we have
	\begin{align*}
		r_{k}(d,n) & = r_{k}(d,n-1)+r_{k-1}(d,n-1)\\
		& \kern1cm + \left(\sum_{i=1}^{n-3}\sum_{j=0}^{k-1}\frac{r_{j}(d,n-i-1)}{i}\binom{i}{k-j}\binom{i}{k-j-1}\right)\\
		& \kern1cm + \binom{n-2}{k-1}\left(\frac{d}{k}\binom{n-1}{k-1}+\frac{1}{n-2}\left(\binom{n-2}{k-2}+\binom{n-2}{k}\right)\right).
	\end{align*}
\end{proposition}
\begin{proof}
	The value $r_{0}(d,n)=1=r_{n}(d,n)$ follows from the fact that $\pnc{G(d,d,n)}{1}$ has a least and a greatest element.  The value $r_{1}(d,n)=(n-1)(n+d-2)=r_{n-1}(d,n)$ follows from Lemmas~\ref{lem:gddn_atoms} and \ref{lem:gddn_coatoms}.  If $k\in\{2,3,\ldots,n-2\}$, then each part of the decomposition \eqref{eq:gddn_first_decomposition} contributes to $r_{k}(d,n)$.  In order to determine these contributions, let us recall a few facts.  First of all, recall from \cite{kreweras72sur}*{Corollaire~4.1} that the rank-generating polynomial of $\pnc{G(1,1,n)}{1}$ is given by
	\begin{displaymath}
		\RR_{\pnc{G(1,1,n)}{1}}(t) = \text{Nar}(n,1)+\text{Nar}(n,2)t+\cdots\text{Nar}(n,n)t^{n-1},
	\end{displaymath}
	where $\text{Nar}(n,k)=\tfrac{1}{n}\binom{n}{k}\binom{n}{k-1}$ is a Narayana number.  Moreover, if $\PP$ and $\QQ$ are graded posets, then the $k\th$ coefficient of $\RR_{\PP\times\QQ}(t)$ is given by $\sum_{j=0}^{k}{r_{j}(\PP)r_{k-j}(\QQ)}$.  It follows from Lemmas~\ref{lem:first_chunk}--\ref{lem:fifth_chunk} that the following elements contribute to the $k\th$ rank of $\pnc{G(d,d,n)}{1}$ for $k\in\{2,3,\ldots,n-2\}$:
	\begin{itemize}
		\item the elements of rank $k$ in $\mathbf{2}\times\pnc{G(d,d,n-1)}{1}$,
		\item $d$ times the elements rank $k-1$ of $\pnc{G(1,1,n-1)}{1}$,
		\item for each $i\in\{3,4,\ldots,n-1\}$ twice the elements of rank $k-1$ in $\pnc{G(d,d,n-i+1)}{1}\times\pnc{G(1,1,i-2)}{1}$,
		\item the elements of rank $k-2$ in $\pnc{G(1,1,n-2)}{1}$, and
		\item the elements of rank $k-1$ in $\pnc{G(1,1,n-2)}{1}$.
	\end{itemize}
	If we put the corresponding numbers together, we obtain
	\begin{align*}
		r_{k}(d,n) & = \sum_{j=0}^{k}{r_{j}(d,n-1)r_{k-j}(\mathbf{2})} + 2\sum_{i=3}^{n-1}\sum_{j=0}^{k-1}{r_{j}(d,n-i+1)r_{k-j-1}\bigl(\pnc{G(1,1,i-2)}{1}\bigr)}\\
		& \kern1cm + d\cdot r_{k-1}\bigl(\pnc{G(1,1,n-1)}{1}\bigr) + r_{k-2}\bigl(\pnc{G(1,1,n-2)}{1}\bigr) + r_{k-1}\bigl(\pnc{G(1,1,n-2)}{1}\bigr).
	\end{align*}
	Since $\mathbf{2}$ has rank $1$, only the terms with $j\in\{k-1,k\}$ contribute to the first sum, and as mentioned before, we have $r_{k}\bigl(\pnc{G(1,1,n)}{1}\bigr)=\text{Nar}(n,k+1)$.  Therefore, the above formula simplifies to
	\begin{align*}
		r_{k}(d,n) & = r_{k-1}(d,n-1)+r_{k}(d,n-1) + 2\sum_{i=3}^{n-1}\sum_{j=0}^{k-1}{r_{j}(d,n-i+1)\text{Nar}(i-2,k-j)}\\
		& \kern1cm + d\cdot\text{Nar}(n-1,k) + \text{Nar}(n-2,k-1) + \text{Nar}(n-2,k).
	\end{align*}
	If we shift the indices in the second sum, and regroup the binomial coefficients appearing in the terms involving Narayana numbers, then we arrive at the desired formula.
\end{proof}

\subsection{A Second Decomposition of $\pnc{G(d,d,n)}{1}(\gamma)$}
	\label{sec:second_decomposition}
We have seen in Section~\ref{sec:first_decomposition} that the decomposition \eqref{eq:gddn_first_decomposition} is not symmetric when $n>2$, which is due to the fact that the parts $R_{1}^{(1)}$ and $R_{2}^{(d-1)}$ do not sit in $\pnc{G(d,d,n)}{1}(\gamma)$ symmetrically, and their induced subposet is disconnected.  We now describe a modification of this decomposition that suits our needs.  Loosely speaking, we peel a copy of $R_{1}^{(1)}$ and $R_{2}^{(d-1)}$ off of $R_{n}^{(d-1)}$, and glue these copies to $R_{1}^{(1)}$ and $R_{2}^{(d-1)}$, respectively.  This produces two symmetric pieces that are each isomorphic to a direct product of a two-chain and a smaller noncrossing partition lattice (see Lemma~\ref{lem:chunk_rearrangement}).  The remaining part of $R_{n}^{(d-1)}$ is symmetrically decomposable into direct products of smaller noncrossing partition lattices by virtue of Theorem~\ref{thm:g11n_decomposition} (see Lemma~\ref{lem:middle_chunk}).

Recall from Lemma~\ref{lem:fifth_chunk} that $\RR_{1}^{(1)}$ has least element $\Bigl[\colint{1}{0}\Bigr]\Bigl[\colint{n}{0}\Bigr]_{d-1}$ and greatest element $\Bigl[\colint{1}{0}\Bigr]\Bigl[\colint{n}{0}\Bigr]_{d-1}\Bigll\colint{2}{0}\;\ldots\;\colint{(n-1)}{0}\Bigrr$.  The map 
\begin{displaymath}
	f_{1}:R_{1}^{(1)}\to R_{n}^{(d-1)},\quad x\mapsto\colref{1}{n}{d-2}x
\end{displaymath}
is clearly injective, since group multiplication is.  Moreover, since $\Bigl[\colint{1}{0}\Bigr]\Bigl[\colint{n}{0}\Bigr]_{d-1}=\colref{1}{n}{d-2}\colref{1}{n}{d-1}$, the image of $f_{1}$ constitutes the interval 
\begin{displaymath}
	E_{1} = \left[\colref{1}{n}{d-1},\colref{1}{n}{d-1}\Bigll\colint{2}{0}\;\ldots\;\colint{(n-1)}{0}\Bigrr\right]_{T}
\end{displaymath}
in $\pnc{G(d,d,n)}{1}(\gamma)$, which is isomorphic to $\pnc{G(1,1,n-2)}{1}$ in its own right.  Lemma~\ref{lem:fifth_chunk} implies now that the restriction $f_{1}:R_{1}^{(1)}\to E_{1}$ is an isomorphism.

Recall further that $\RR_{2}^{(d-1)}$ has least element $\colref{1}{2}{d-1}$ and greatest element $\Bigll(\colint{1}{0}\;\colint{2}{d-1}\;\ldots\;\colint{(n-1)}{d-1}\Bigrr$.  The map
\begin{displaymath}
	f_{2}:R_{2}^{(d-1)}\to R_{n}^{(d-1)},\quad x\mapsto\colref{2}{n}{0}x
\end{displaymath}
is injective, and since $\colref{2}{n}{0}\colref{1}{2}{d-1}=\Bigll\colint{1}{0}\;\colint{n}{d-1}\;\colint{2}{d-1}\Bigrr$ the image of $f_{2}$ constitutes the interval 
\begin{displaymath}
	E_{2} = \left[\Bigll\colint{1}{0}\;\colint{n}{d-1}\;\colint{2}{d-1}\Bigrr,\Bigll\colint{1}{0}\;\colint{n}{d-1}\;\colint{2}{d-1}\;\ldots\;\colint{(n-1)}{d-1}\Bigrr\right]_{T}
\end{displaymath}
in $\pnc{G(d,d,n)}{1}(\gamma)$, which is again isomorphic to $\pnc{G(1,1,n-2)}{1}$.  Lemma~\ref{lem:fifth_chunk} implies again that the restriction $f_{2}:R_{2}^{(d-1)}\to E_{2}$ is an isomorphism.  Define
\begin{align*}
	D_{1} & = R_{1}^{(1)}\uplus E_{1},\\
	D_{2} & = R_{2}^{(d-1)}\uplus E_{2},\quad\text{and}\\
	D & = R_{n}^{(d-1)}\setminus\Bigl(E_{1}\uplus E_{2}\Bigr),
\end{align*}
as well as $\DD_{1}=\bigl(D_{1},\leq_{T}\bigr), \DD_{2}=\bigl(D_{2},\leq_{T}\bigr)$, and $\DD=\bigl(D,\leq_{T})$.  The following lemma is now immediate.

\begin{lemma}\label{lem:chunk_rearrangement}
	For $d,n\geq 2$, we have $D_{1}\uplus D_{2}\uplus D = R_{1}^{(1)}\uplus R_{2}^{(d-1)}\uplus R_{n}^{(d-1)}$.  Moreover, $D_{1}\cong D_{2}\cong\mathbf{2}\times\pnc{G(1,1,n-2)}{1}$, and the least elements of $\DD_{1}$ and $\DD_{2}$ have length $1$, while the greatest elements of $\DD_{1}$ and $\DD_{2}$ have length $n-1$.
\end{lemma}
\begin{proof}
	It is clear by definition that $D_{1}\uplus D_{2}\uplus D = R_{1}^{(1)}\uplus R_{2}^{(d-1)}\uplus R_{n}^{(d-1)}$.  In view of Lemma~\ref{lem:bottom_intervals}, the maps $f_{1}$ and $f_{2}$ are poset isomorphisms, when we restrict their images to $E_{1}$ and $E_{2}$, respectively.  It follows now from Lemma~\ref{lem:fifth_chunk} that $D_{1}\cong D_{2}\cong\mathbf{2}\times\pnc{G(1,1,n-2)}{1}$.  
	
	The least element of $\DD_{1}$ is $\colref{1}{n}{d-1}$, while the greatest element of $\DD_{1}$ is $\Bigl[\colint{1}{0}\Bigr]\Bigl[\colint{n}{0}\Bigr]_{d-1}\Bigll\colint{2}{0}\;\ldots\;\colint{(n-1)}{0}\Bigrr$.  Analogously, the least element of $\DD_{2}$ is $\colref{1}{2}{d-1}$, while the greatest element of $\DD_{2}$ is $\Bigll(\colint{1}{0}\;\colint{n}{d-1}\;\colint{2}{d-1}\;\ldots\;\colint{(n-1)}{d-1}\Bigrr$, which completes the proof.
\end{proof}

\begin{lemma}\label{lem:middle_chunk}
	For $d,n\geq 2$, we have 
	\begin{displaymath}
		D \cong \biguplus_{i=3}^{n-1}{\left(\nc{G(1,1,i-2)}{1}\times\nc{G(1,1,n-i)}{1}\right)}.
	\end{displaymath}	
	Moreover, this decomposition is symmetric as a poset decomposition.
\end{lemma}
\begin{proof}
	It follows from Lemma~\ref{lem:second_chunk} that the map $x\mapsto\colref{1}{n}{(d-1)}x$ is a bijection from $R_{n}^{(d-1)}$ to
	\begin{displaymath}
		\Bigl\{x\in\nc{G(d,d,n)}{1}(\gamma)\mid x\leq_{T}\Bigll\colint{2}{0}\;\ldots\;\colint{n}{0}\Bigrr\Bigr\},
	\end{displaymath}
	and this set is in bijection with $\nc{G(1,1,n-1)}{1}\bigl((1\;2\;\ldots\;(n-1)\bigr)$.  Under this bijection, the elements in $E_{1}=f_{1}\Bigl(R_{1}^{(1)}\Bigr)$ correspond to the permutations in $R_{1}$, while the elements in $E_{2}=f_{2}\Bigl(R_{2}^{(d-1)}\Bigr)$ correspond to the permutations in $R_{2}$.  Now the claim follows from Theorem~\ref{thm:g11n_decomposition}.  
\end{proof}

\begin{theorem}\label{thm:gddn_symmetric_decomposition}
	For $d,n\geq 2$ the decomposition 
	\begin{equation}\label{eq:gddn_second_decomposition}
		\nc{G(d,d,n)}{1}(\gamma) = R_{1}^{(0)}\uplus R_{2}^{(0)}\uplus\biguplus_{i=3}^{n-1}\left(R_{i}^{(0)}\uplus R_{i}^{(d-1)}\right)\uplus\biguplus_{s=0}^{d-2}{R_{n}^{(s)}}\uplus D_{1}\uplus D_{2}\uplus D
	\end{equation}
	is symmetric.
\end{theorem}
\begin{proof}
	This follows from Lemmas~\ref{lem:first_chunk}--\ref{lem:fourth_chunk} and Lemmas~\ref{lem:chunk_rearrangement} and \ref{lem:middle_chunk}.  
\end{proof}

In particular, the symmetric decomposition of $\pnc{G(d,d,n)}{1}(\gamma)$ in Theorem~\ref{thm:gddn_symmetric_decomposition} consists of parts that are direct products of noncrossing partition lattices of smaller rank.  As a consequence, we can prove Theorem~\ref{thm:gddn_sbd} by induction on rank.  The next lemma states that the induction hypthesis holds. 

\begin{lemma}\label{lem:induction_base}
	For $d\geq 2$, the lattices $\pnc{G(d,d,2)}{1}$ and $\pnc{G(d,d,3)}{1}$ admit a symmetric Boolean decomposition. 
\end{lemma}
\begin{proof}
	Recall that for $d\geq 2$, the lattice $\pnc{G(d,d,2)}{1}$ has rank $2$ and exactly $d$ atoms.  Therefore, it trivially admits a symmetric Boolean decomposition.  (Fix one Boolean lattice of rank $2$, involving the least and the greatest element, and consider all the remaining atoms as Boolean lattices of rank $0$.)  
	
	Now consider the Coxeter element $\gamma\in G(d,d,3)$ as defined in \eqref{eq:gddn_coxeter_element}, as well as the decomposition of $\nc{G(d,d,3)}{1}(\gamma)$ from \eqref{eq:gddn_second_decomposition}, which can be simplified to
	\begin{displaymath}
		\nc{G(d,d,3)}{1}(\gamma) = R_{1}^{(0)}\uplus R_{2}^{(0)}\uplus\biguplus_{s=0}^{d-2}{R_{3}^{(s)}}\uplus D_{1}\uplus D_{2}\uplus D.
	\end{displaymath}	
	Lemma~\ref{lem:first_chunk} implies that the subposet $\RR_{1}^{(0)}\uplus\RR_{2}^{(0)}$ is isomorphic to $\mathbf{2}\times\pnc{G(d,d,2)}{1}$, and the first part of this proof, together with the fact that symmetric Boolean decompositions are preserved under direct products~\cite{petersen15eulerian}*{Observation~4.4}, implies that we can decompose this part into symmetric Boolean lattices.  Lemma~\ref{lem:second_chunk} implies that $\RR_{3}^{(s)}\cong\pnc{G(1,1,2)}{1}\cong\mathbf{2}$ for $0\leq s<d-1$, which is the Boolean lattice of rank $1$.  Lemma~\ref{lem:chunk_rearrangement} implies that $\DD_{1}\cong\DD_{2}\cong\mathbf{2}$, and Lemma~\ref{lem:middle_chunk} implies that $D=\emptyset$, and we are done.  
	
	Since Proposition~\ref{prop:noncrossing_partitions_independent} implies that the structure of $\pnc{G(d,d,n)}{1}$ does not depend on the choice of Coxeter element, we conclude that $\pnc{G(d,d,3)}{1}$ admits a symmetric Boolean decomposition for any choice of Coxeter element.
\end{proof}

\begin{proof}[Proof of Theorem~\ref{thm:gddn_sbd}]
	First of all we consider $\pnc{G(d,d,n)}{1}(\gamma)$, where $\gamma\in G(d,d,n)$ is the Coxeter element defined in \eqref{eq:gddn_coxeter_element}.  Lemma~\ref{lem:induction_base} implies that we can assume that $\pnc{G(d,d,n')}{1}(\gamma')$ admits a symmetric Boolean decomposition for every $n'<n$ and every parabolic Coxeter element $\gamma'\in G(d,d,n')$.  We further use the fact that symmetric Boolean decompositions are preserved under direct products~\cite{petersen15eulerian}*{Observation~4.4} and Corollary~\ref{cor:g11n_sbd}, as well as Lemmas~\ref{lem:first_chunk}--\ref{lem:fourth_chunk} and Lemmas~\ref{lem:chunk_rearrangement} and \ref{lem:middle_chunk} to conclude that the decomposition of $\pnc{G(d,d,n)}{1}(\gamma)$ described in \eqref{eq:gddn_second_decomposition} consists of pieces all of which admit a symmetric Boolean decomposition, and we are done.
	
	Once more we can generalize the existence of a symmetric Boolean decomposition from $\pnc{G(d,d,n)}{1}(\gamma)$ to every Coxeter element of $G(d,d,n)$ using Proposition~\ref{prop:noncrossing_partitions_independent}.  Propositions~\ref{prop:sbd_implies_scd}--\ref{prop:symmetric_boolean_gamma_nonnegative} conclude the proof.
\end{proof}

\begin{amssidewaysfigure}
	\centering
	\begin{tikzpicture}\small
		\def\x{1.75};
		\def\y{5};
		\def\s{.67};
		\draw(6.5*\x,.5*\y) node[scale=\s](n1){$\Bigll\colint{1}{0}\Bigrr$};
		\draw(1*\x,1*\y) node[scale=\s](n2){$\Bigll\colint{2}{0}\;\colint{3}{0}\Bigrr$};
		\draw(2*\x,1*\y) node[scale=\s](n3){$\Bigll\colint{1}{0}\;\colint{3}{0}\Bigrr$};
		\draw(3*\x,1*\y) node[scale=\s](n4){$\Bigll\colint{1}{0}\;\colint{2}{4}\Bigrr$};
		\draw(4*\x,1*\y) node[scale=\s](n5){$\Bigll\colint{1}{0}\;\colint{3}{4}\Bigrr$};
		\draw(5*\x,1*\y) node[scale=\s](n6){$\Bigll\colint{2}{0}\;\colint{3}{1}\Bigrr$};
		\draw(6*\x,1*\y) node[scale=\s](n7){$\Bigll\colint{2}{0}\;\colint{3}{4}\Bigrr$};
		\draw(7*\x,1*\y) node[scale=\s](n8){$\Bigll\colint{1}{0}\;\colint{3}{1}\Bigrr$};
		\draw(8*\x,1*\y) node[scale=\s](n9){$\Bigll\colint{1}{0}\;\colint{3}{3}\Bigrr$};
		\draw(9*\x,1*\y) node[scale=\s](n10){$\Bigll\colint{2}{0}\;\colint{3}{2}\Bigrr$};
		\draw(10*\x,1*\y) node[scale=\s](n11){$\Bigll\colint{1}{0}\;\colint{2}{0}\Bigrr$};
		\draw(11*\x,1*\y) node[scale=\s](n12){$\Bigll\colint{2}{0}\;\colint{3}{3}\Bigrr$};
		\draw(12*\x,1*\y) node[scale=\s](n13){$\Bigll\colint{1}{0}\;\colint{3}{2}\Bigrr$};
		\draw(1*\x,2*\y) node[scale=\s](n14){$\Bigll\colint{1}{0}\;\colint{2}{0}\;\colint{3}{0}\Bigrr$};
		\draw(2*\x,2*\y) node[scale=\s](n15){$\Bigll\colint{1}{0}\;\colint{3}{4}\;\colint{2}{4}\Bigrr$};
		\draw(3*\x,2*\y) node[scale=\s](n16){$\Bigl[\colint{1}{0}\Bigr]\Bigl[\colint{3}{0}\Bigr]_{4}$};
		\draw(4*\x,2*\y) node[scale=\s](n17){$\Bigll\colint{1}{0}\;\colint{3}{0}\;\colint{2}{4}\Bigrr$};
		\draw(5*\x,2*\y) node[scale=\s](n18){$\Bigll\colint{1}{0}\;\colint{2}{0}\;\colint{3}{4}\Bigrr$};
		\draw(6*\x,2*\y) node[scale=\s](n19){$\Bigll\colint{1}{0}\;\colint{2}{0}\;\colint{3}{1}\Bigrr$};
		\draw(7*\x,2*\y) node[scale=\s](n20){$\Bigll\colint{1}{0}\;\colint{3}{3}\;\colint{2}{4}\Bigrr$};
		\draw(8*\x,2*\y) node[scale=\s](n21){$\Bigll\colint{1}{0}\;\colint{3}{1}\;\colint{2}{4}\Bigrr$};
		\draw(9*\x,2*\y) node[scale=\s](n22){$\Bigll\colint{1}{0}\;\colint{2}{0}\;\colint{3}{3}\Bigrr$};
		\draw(10*\x,2*\y) node[scale=\s](n23){$\Bigl[\colint{2}{0}\Bigr]\Bigl[\colint{3}{0}\Bigr]_{4}$};
		\draw(11*\x,2*\y) node[scale=\s](n24){$\Bigll\colint{1}{0}\;\colint{2}{0}\;\colint{3}{2}\Bigrr$};
		\draw(12*\x,2*\y) node[scale=\s](n25){$\Bigll\colint{1}{0}\;\colint{3}{2}\;\colint{2}{4}\Bigrr$};
		\draw(6.5*\x,2.5*\y) node[scale=\s](n26){$\Bigl[\colint{1}{0}\;\colint{2}{0}\Bigr]\Bigl[\colint{3}{0}\Bigr]_{4}$};
		\draw(n1) -- (n2);
		\draw(n1) -- (n3);
		\draw(n1) -- (n4);
		\draw(n1) -- (n5);
		\draw(n1) -- (n6);
		\draw(n1) -- (n7);
		\draw(n1) -- (n8);
		\draw(n1) -- (n9);
		\draw(n1) -- (n10);
		\draw(n1) -- (n11);
		\draw(n1) -- (n12);
		\draw(n1) -- (n13);
		\draw(n2) -- (n14);
		\draw(n2) -- (n15);
		\draw(n2) -- (n23);
		\draw(n3) -- (n14);
		\draw(n3) -- (n16);
		\draw(n3) -- (n17);
		\draw(n4) -- (n15);
		\draw(n4) -- (n17);
		\draw(n4) -- (n20);
		\draw(n4) -- (n21);
		\draw(n4) -- (n25);
		\draw(n5) -- (n15);
		\draw(n5) -- (n16);
		\draw(n5) -- (n18);
		\draw(n6) -- (n17);
		\draw(n6) -- (n19);
		\draw(n6) -- (n23);
		\draw(n7) -- (n18);
		\draw(n7) -- (n20);
		\draw(n7) -- (n23);
		\draw(n8) -- (n16);
		\draw(n8) -- (n19);
		\draw(n8) -- (n21);
		\draw(n9) -- (n16);
		\draw(n9) -- (n20);
		\draw(n9) -- (n22);
		\draw(n10) -- (n21);
		\draw(n10) -- (n23);
		\draw(n10) -- (n24);
		\draw(n11) -- (n14);
		\draw(n11) -- (n18);
		\draw(n11) -- (n19);
		\draw(n11) -- (n22);
		\draw(n11) -- (n24);
		\draw(n12) -- (n22);
		\draw(n12) -- (n23);
		\draw(n12) -- (n25);
		\draw(n13) -- (n16);
		\draw(n13) -- (n24);
		\draw(n13) -- (n25);
		\draw(n14) -- (n26);
		\draw(n15) -- (n26);
		\draw(n16) -- (n26);
		\draw(n17) -- (n26);
		\draw(n18) -- (n26);
		\draw(n19) -- (n26);
		\draw(n20) -- (n26);
		\draw(n21) -- (n26);
		\draw(n22) -- (n26);
		\draw(n23) -- (n26);
		\draw(n24) -- (n26);
		\draw(n25) -- (n26);
		\begin{pgfonlayer}{background}
			\draw[rounded corners,orange!50!black,opacity=.4,cap=round,line width=5pt](n1) -- (n2);
			\draw[rounded corners,orange!50!black,opacity=.4,cap=round,line width=5pt](n1) -- (n7);
			\draw[rounded corners,orange!50!black,opacity=.4,cap=round,line width=5pt](n1) -- (n11);
			\draw[rounded corners,orange!50!black,opacity=.4,cap=round,line width=5pt](n2) -- (n14);
			\draw[rounded corners,orange!50!black,opacity=.4,cap=round,line width=5pt](n2) -- (n23);
			\draw[rounded corners,orange!50!black,opacity=.4,cap=round,line width=5pt](n7) -- (n18);
			\draw[rounded corners,orange!50!black,opacity=.4,cap=round,line width=5pt](n7) -- (n23);
			\draw[rounded corners,orange!50!black,opacity=.4,cap=round,line width=5pt](n11) -- (n14);
			\draw[rounded corners,orange!50!black,opacity=.4,cap=round,line width=5pt](n11) -- (n18);
			\draw[rounded corners,orange!50!black,opacity=.4,cap=round,line width=5pt](n14) -- (n26);
			\draw[rounded corners,orange!50!black,opacity=.4,cap=round,line width=5pt](n18) -- (n26);
			\draw[rounded corners,orange!50!black,opacity=.4,cap=round,line width=5pt](n23) -- (n26);
			\draw[rounded corners,green!50!black,opacity=.4,cap=round,line width=5pt](n3) -- (n17);
			\draw[rounded corners,green!50!black,opacity=.4,cap=round,line width=5pt](n4) -- (n15);
			\draw[rounded corners,green!50!black,opacity=.4,cap=round,line width=5pt](n5) -- (n16);
			\draw[rounded corners,green!50!black,opacity=.4,cap=round,line width=5pt](n6) -- (n19);
			\draw[rounded corners,green!50!black,opacity=.4,cap=round,line width=5pt](n8) -- (n21);
			\draw[rounded corners,green!50!black,opacity=.4,cap=round,line width=5pt](n9) -- (n20);
			\draw[rounded corners,green!50!black,opacity=.4,cap=round,line width=5pt](n10) -- (n24);
			\draw[rounded corners,green!50!black,opacity=.4,cap=round,line width=5pt](n12) -- (n22);
			\draw[rounded corners,green!50!black,opacity=.4,cap=round,line width=5pt](n13) -- (n25);
		\end{pgfonlayer}
	\end{tikzpicture}
	\caption{The lattice $\pnc{G(5,5,3)}{1}\Bigl(\Bigl[\colint{1}{0}\;\colint{2}{0}\Bigr]\Bigl[\colint{3}{0}\Bigr]_{4}\Bigr)$ again.  A symmetric Boolean decomposition coming from \eqref{eq:gddn_second_decomposition} is highlighted, where the single Boolean lattice of rank $2$ is colored in brown, while the remaining nine Boolean lattices of rank $1$ are colored in green.}
	\label{fig:g553_nc_chain_decomposition}
\end{amssidewaysfigure}

\section{The Remaining Cases}
	\label{sec:remaining_cases}
Now that we have completed the affirmative answer of Questions~\ref{qu:nc_sbd_scd} and \ref{qu:nc_ssp} for the infinite families of irreducible well-generated complex reflection groups, it remains to deal with the exceptional cases.  These could in principle be treated by computer, and we strongly believe that Question~\ref{qu:nc_sbd_scd} has an affirmative answer for these groups as well.  However, explicitly constructing such a decomposition is computationally hard.  In this section we provide some computational evidence that supports this belief.  In particular, we show that the consequences of the existence of a symmetric Boolean decomposition mentioned in Section~\ref{sec:posets}, namely the strong Sperner property (which affirmatively answers Question~\ref{qu:nc_ssp} for all well-generated complex reflection groups), rank-symmetry, rank-unimodality, and rank-$\gamma$-nonnegativity, hold in each noncrossing partition lattice.  

Before we accomplish that, we use another tool to show the existence of a symmetric chain decomposition for noncrossing partition lattices of small rank.  Let $\PP=(P,\leq)$ be a graded poset of rank $n$ with rank-generating polynomial $\RR_{\PP}(t)=r_{0}(\PP)+r_{1}(\PP)t+\cdots+r_{n}(\PP)t^n$, and let $R_{i}=\{x\in P\mid\text{rk}(x)=i\}$ for $i\in\{0,1,\ldots,n\}$.  If $L\subseteq R_{i}$, then we define its \alert{shade} to be $\nabla L=\{p\in R_{i+1}\mid x\lessdot p\;\text{for some}\;x\in L\}$.  We say that $\PP$ \alert{has the normalized matching property} if for every $i\in\{0,1,\ldots,n\}$, and every $L\subseteq R_{i}$ the following inequality is satisfied:
\begin{equation}\label{eq:normalized}
	\frac{\bigl\lvert\nabla L\bigr\rvert}{r_{i+1}(\PP)}\geq\frac{\bigl\lvert L\bigr\rvert}{r_{i}(\PP)}.
\end{equation}

\begin{theorem}[\cite{griggs77sufficient}]\label{thm:normalized_matching}
	If $\PP$ is a graded, rank-symmetric, and rank-unimodal poset that has the normalized matching property, then $\PP$ admits a symmetric chain decomposition.
\end{theorem}

We have the following result.

\begin{proposition}\label{prop:scd_small_ranks}
	If $W$ is an exceptional irreducible well-generated complex reflection group of rank $\leq 4$, then $\pnc{W}{1}$ admits a symmetric chain decomposition (except possibly for $W=G_{30}$).
\end{proposition}
\begin{proof}
	This was verified using \textsc{Sage}~\cites{sage,sagecombinat} and \textsc{Gap}~\cite{michel15development}.  More precisely, \textsc{Gap} was used to create the noncrossing partition lattices, and \textsc{Sage} was used to verify that these lattices satisfy the assumptions of Theorem~\ref{thm:normalized_matching}.  
\end{proof}

%\begin{remark}
%	We could explicitly check the normalized matching property only for those noncrossing partition lattices associated with exceptional well-generated complex reflection groups of rank at most $3$.  In rank $4$, we applied the following strategy suggested by a referee.
%	
%	Observe that if two sets $X$ and $X'$ have the same shade, then it suffices to check \eqref{eq:normalized} for the larger of the two sets.  It follows immediately from the definition that noncrossing partition lattices are rank-symmetric.  Therefore, in rank $4$ we can avoid running through all sets of elements having rank $2$, and instead consider any set $X$ consisting of elements of rank $3$ and compute the maximal set $X'$ that has shade $X$, and check \eqref{eq:normalized} for these sets.  This was feasible for the lattices $\pnc{G_{28}}{1}$, $\pnc{G_{29}}{1}$, and $\pnc{G_{32}}{1}$.  
%	
%	The lattice $\pnc{G_{30}}{1}$ (which is isomorphic to the noncrossing partition lattice associated with the Coxeter group $H_{4}$), however, has $60$ atoms and $60$ coatoms, so that this strategy is not applicable. 
%\end{remark}

\begin{remark}
	A particularly rewarding answer to Question~\ref{qu:nc_sbd_scd} would be one that involves a uniform argument, \ie an argument that does not rely on the classification of irreducible well-generated complex reflection groups.  
	
	We observe that the strategies producing symmetric Boolean decompositions of the noncrossing partition lattices associated with each of the groups $G(1,1,n)$, $G(d,1,n)$, and $G(d,d,n)$ reviewed in this article, all rely on a decomposition of these lattices into parts that correspond to direct products of noncrossing partition lattices associated with parabolic subgroups of either of these groups.  If we try to do the same for $\pnc{G_{28}}{1}$ (which is isomorphic to the noncrossing partition lattice associated with the Coxeter group $F_{4}$), then this process fails.  In particular, if we remove a subposet from $\pnc{G_{28}}{1}$ that corresponds to a direct product of a $2$-chain and a noncrossing partition lattice associated with a maximal parabolic subgroup of $G_{28}$, then the resulting poset does not admit a symmetric chain decomposition anymore.  
	
	Therefore, a potential uniform proof answering Question~\ref{qu:nc_sbd_scd} affirmatively needs to pursue a completely different approach.  One possible approach might be a uniform definition of an R*S-labeling of these posets, since \cite{hersh99deformation}*{Theorem~5} implies that each graded poset with such a labeling admits a symmetric Boolean decomposition.
\end{remark}

While it may not be feasible to run through all subsets of every single rank of a poset, there exist fast algorithms to compute the maximum size of an antichain (or the \alert{width}) of a given poset, and it is thus possible to check by computer if a poset is Sperner.  In what follows we describe a strategy to reduce the question whether a poset is strongly Sperner to successively checking whether certain subposets are Sperner.  

Let $\PP=(P,\leq)$ be a graded poset of rank $n$ with rank-generating polynomial $\RR_{\PP}(t)=r_{0}(\PP)+r_{1}(\PP)t+\cdots+r_{n}(\PP)t^{n}$.  There is certainly some $s\in\{0,1,\ldots,n\}$ such that $r_{j}(\PP)\leq r_{s}(\PP)$ for all $j\in\{0,1,\ldots,n\}$.  (This $s$ need not be unique.)  Let $X=\{p\in P\mid\text{rk}(p)=s\}$, and define $\PP[1]=(P\setminus X,\leq)$.  Moreover, define $\PP[0]=\PP$ and $\PP[i]=\bigl(\ldots\bigl(\bigl(\PP\underbrace{[1]\bigr)[1]\bigr)\ldots\bigr)[1]}_{i}$.  In other words, $\PP[i]$ is the poset that is created from $\PP$ by removing the $i$ largest ranks, and is consequently a graded poset in its own right.  Clearly, $\PP[s]$ is the empty poset for $s>n$.  

\begin{proposition}\label{prop:strongly_sperner_rank_removal}
	A graded poset $\PP$ of rank $n$ is strongly Sperner if and only if $\PP[i]$ is Sperner for each $i\in\{0,1,\ldots,n\}$.
\end{proposition}
\begin{proof}
	First of all, we want to emphasize that for each $i$ the poset $\PP[i]$ is an induced subposet of $\PP$, so every antichain in $\PP[i]$ is also an antichain in $\PP$.  Moreover, for each $i\in\{0,1,\ldots,n\}$, let $s_{i}\in\{0,1,\ldots,n-i\}$ be such that $r_{j}(\PP[i])\leq r_{s_{i}}(\PP[i])$ for all $j\in\{0,1,\ldots,n-i\}$, and write $X^{(i)}=\bigl\{p\in P[i]\mid \rk(p)=s_{i}\bigr\}$, where $P[i]$ denotes the ground set of $\PP[i]$.  

	\smallskip

	Now suppose that $\PP$ is strongly Sperner, and there is a minimal index $i\leq n$ such that $\PP[i]$ is not Sperner.  Since $\PP$ is strongly Sperner, it follows that $i>0$.  Consequently, we can find an antichain $A$ in $\PP[i]$ whose size exceeds the largest rank size of $\PP[i]$.  It follows that $A\uplus R^{(i-1)}$ is a $2$-family in $\PP[i-1]$ whose size exceeds the sum of the two largest rank numbers.  It follows further that $A\uplus R^{(i-1)}\uplus R^{(i-2)}\uplus\cdots\uplus R^{(0)}$ is an $(i+1)$-family in $\PP[0]=\PP$ whose size exceeds the sum of the $i+1$ largest rank numbers, which contradicts the assumption that $\PP$ is strongly Sperner. 
	
	\smallskip
	
	Conversely, suppose that each of $\PP[0],\PP[1],\ldots,\PP[n]$ is Sperner.  Since $\PP[0]=\PP$ it follows that $\PP$ is $1$-Sperner.  Thus a maximum-sized antichain in $\PP$ is for instance $R^{(0)}$.  Since $\PP[1]$ is Sperner, a maximum-sized antichain in $\PP[1]$ is for instance $R^{(1)}$, or in other words, a maximum-sized $2$-family in $\PP$ is for instance $R^{(0)}\uplus R^{(1)}$.  It follows that $\PP$ is $2$-Sperner.  The same argument shows that for every $i\in[n]$ a maximum-sized $i$-family in $\PP$ is for instance $R^{(0)}\uplus R^{(1)}\uplus\cdots\uplus R^{(i-1)}$, and hence $\PP$ is $i$-Sperner.  By definition, it follows that $\PP$ is strongly Sperner.
\end{proof}

The weaker statement ``$\PP$ is $k$-Sperner if and only if $\PP[i]$ is Sperner for each $i\in\{0,1,\ldots,k\}$'' is not true, as for instance the example in Figure~\ref{fig:sperner_posets} shows.  If the poset in Figure~\ref{fig:sperner_posets_2} is denoted by $\PP$, then the poset in Figure~\ref{fig:sperner_posets_3} is $\PP[1]$.  We observe that $\PP$ is $1$-Sperner, but $\PP[1]$ is not.  Now we conclude the proof of Theorem~\ref{thm:nc_strongly_sperner}.

\begin{table}
	\begin{tabular}{c|c|c}
		$W$ & rank vector of $\pnc{W}{1}$ & $\gamma$-vector of $\pnc{W}{1}$ \\
		\hline
		$G_{4}$ & $(1,3,1)$ & $(1,1)$ \\
		$G_{5}$ & $(1,4,1)$ & $(1,2)$ \\
		$G_{6}$ & $(1,6,1)$ & $(1,4)$ \\
		$G_{8}$ & $(1,3,1)$ & $(1,1)$ \\
		$G_{9}$ & $(1,6,1)$ & $(1,4)$ \\
		$G_{10}$ & $(1,4,1)$ & $(1,2)$ \\
		$G_{14}$ & $(1,8,1)$ & $(1,6)$ \\
		$G_{16}$ & $(1,3,1)$ & $(1,1)$ \\
		$G_{17}$ & $(1,6,1)$ & $(1,4)$ \\
		$G_{18}$ & $(1,4,1)$ & $(1,2)$ \\
		$G_{20}$ & $(1,5,1)$ & $(1,3)$ \\
		$G_{21}$ & $(1,10,1)$ & $(1,8)$ \\
		$G_{23}$ & $(1,15,15,1)$ & $(1,12)$ \\
		$G_{24}$ & $(1,14,14,1)$ & $(1,11)$ \\
		$G_{25}$ & $(1,6,6,1)$ & $(1,3)$ \\
		$G_{26}$ & $(1,9,9,1)$ & $(1,6)$ \\
		$G_{27}$ & $(1,20,20,1)$ & $(1,17)$ \\
		$G_{28}$ & $(1,24,55,24,1)$ & $(1,20,9)$ \\
		$G_{29}$ & $(1,25,60,25,1))$ & $(1,21,12)$ \\
		$G_{30}$ & $(1,60,158,60,1)$ & $(1,56,40)$ \\
		$G_{32}$ & $(1,10,20,10,1)$ & $(1,6,2)$ \\
		$G_{33}$ & $(1,30,123,123,30,1)$ & $(1,25,38)$ \\
		$G_{34}$ & $(1,56,385,700,385,56,1)$ & $(1,50,170,40)$ \\
		$G_{35}$ & $(1,36,204,351,204,36,1)$ & $(1,30,69,13)$ \\
		$G_{36}$ & $(1,63,546,1470,1470,546,63,1)$ & $(1,56,245,140)$ \\
		$G_{37}$ & $(1,120,1540,6120,9518,6120,1540,120,1)$ & $(1,112,840,1024,120)$ \\
	\end{tabular}
	\caption{The rank vectors and $\gamma$-vectors of the noncrossing partition lattices associated with exceptional irreducible well-generated complex reflection groups.}
	\label{tab:exceptional_rank_numbers}
\end{table}

\begin{proof}[Proof of Theorem~\ref{thm:nc_strongly_sperner}]
	First suppose that $W$ is irreducible.  Propositions~\ref{prop:symmetric_chains_strongly_sperner} and~\ref{prop:symmetric_boolean_gamma_nonnegative} imply together with Corollary~\ref{cor:g11n_sbd}, \cite{hersh99deformation}*{Theorems~5~and~7}, and Theorem~\ref{thm:gddn_sbd} that $\pnc{W}{1}$ is strongly Sperner, rank-symmetric, rank-unimodal, rank-$\gamma$-nonnegative whenever $W$ is isomorphic to $G(1,1,n)$ for $n\geq 1$, to $G(d,1,n)$ for $d\geq 2,n\geq 1$, or to $G(d,d,n)$ for $d,n\geq 2$.  (Recall that $\pnc{G(d,1,n)}{1}$ is isomorphic to $\pnc{G(2,1,n)}{1}$ for each $d\geq 2$.)  The exceptional cases have been verified to be strongly Sperner by computer using Proposition~\ref{prop:strongly_sperner_rank_removal}, and their rank- and $\gamma$-vectors are listed in Table~\ref{tab:exceptional_rank_numbers}.
	
	If $W$ is reducible, then we have $W\cong W_{1}\times W_{2}\times\cdots\times W_{s}$ for some irreducible well-generated complex reflection groups $W_{i}$, and we simply define $\pnc{W}{1}=\pnc{W_{1}}{1}\times\pnc{W_{2}}{1}\times\cdots\times\pnc{W_{s}}{1}$.  The first part of this proof asserts that each of these factors is a strongly Sperner, rank-symmetric, and rank-unimodal lattice, and then for instance \cite{proctor80product}*{Theorem~3.2} implies that the same is true for $\pnc{W}{1}$.  The rank-$\gamma$-nonnegativity in this case follows from the fact that the rank-generating polynomial of a direct product of posets is the product of the rank-generating polynomials of its factors, and \cite{petersen15eulerian}*{Observation~4.1}, which states that the product of $\gamma$-nonnegative polynomials is again $\gamma$-nonnegative.
\end{proof}

The \textsc{Gap} script that was used to create the noncrossing partition lattices for the well-generated complex reflection groups, and to print them into a format that allows for an import to \textsc{Sage} can be found here: \url{http://homepage.univie.ac.at/henri.muehle/files/print_nc.gap}.  The corresponding files for the noncrossing partition lattices associated with irreducible real reflection groups of rank at most $8$ (except those associated with the dihedral groups, since they are trivial), for the noncrossing partition lattices associated with exceptional irreducible well-generated complex reflection groups, as well as for the noncrossing partition lattices associated with the groups $G(d,d,n)$ for $3\leq d\leq 8$ and $3\leq n\leq 6$, as well as $3\leq d\leq 6$ and $n=7$ can be found here: \url{http://homepage.univie.ac.at/henri.muehle/files/ncp.zip}.  The \textsc{Sage} script which converts the \textsc{Gap} output to a \textsc{Sage} poset object, and which provides functions for checking whether a poset has the normalized matching or the strong Sperner property can be found here: \url{http://homepage.univie.ac.at/henri.muehle/files/sperner.sage}.  

\section{Acknowledgements}
	\label{sec:acknowledgements}
I would like to thank the anonymous referees for many suggestions on how to improve the exposition of this paper, in particular for suggesting a strategy to extend Proposition~\ref{prop:scd_small_ranks}.  The initial version of this paper considered only symmetric chain decompositions of noncrossing partition lattices, and I am indebted to Patricia Hersh and Marko Thiel for asking for symmetric Boolean decompositions of these lattices.  

\appendix

\section{The Proofs of Lemmas~\ref{lem:empty_chunks} and \ref{lem:fifth_chunk}}
	\label{app:fixed_space}
One of the nice properties of the complex reflection group $G(d,d,n)$ is the fact that the absolute length of its noncrossing partitions has an equivalent formulation in terms of the codimension of the fixed space.  Recall that for $w\in G(d,d,n)$ the fixed space is defined by $\text{Fix}(w)=\bigl\{\mathbf{v}\in V\mid w\mathbf{v}=\mathbf{v}\bigr\}$.  We have the following result.

\begin{lemma}[\cite{bessis06non}*{Lemma~4.1(ii)}]\label{lem:gddn_absolute_length_fixed_codimension}
	Let $h$ be the largest degree of $G(d,d,n)$ and let $\gamma$ be a $e^{2i\pi/h}$-regular element.  For every $w\in\nc{G(d,d,n)}{1}(\gamma)$ we have $\ell_{T}(w)=n-\dim\text{Fix}(w)$.
\end{lemma}

We use this connection to prove the following results.

\begin{lemma}\label{lem:fixed_space_first_lemma}
	Let $d,n\geq 2$, and let $\gamma\in G(d,d,n)$ be the Coxeter element defined in \eqref{eq:gddn_coxeter_element}.  We have $\Bigl[\colint{1}{0}\Bigr]_{s}\Bigl[\colint{a}{0}\Bigr]_{d-s}\leq_{T}\gamma$ if and only if $a=n$ and $s=1$.
\end{lemma}
\begin{proof}
	Let $w=\Bigl[\colint{1}{0}\Bigr]_{s}\Bigl[\colint{a}{0}\Bigr]_{d-s}=\colref{1}{a}{0}\colref{1}{a}{s}$ for $1<a\leq n$ and $0\leq s<d$.  In view of \eqref{eq:absolute_order} and Lemma~\ref{lem:gddn_absolute_length_fixed_codimension} it suffices to show that $\dim\text{Fix}(w^{-1}\gamma)=2$ if and only if $a=n$ and $s=1$.
	
	Let $V=\mathbb{C}^{n}$ and pick $\mathbf{v}\in V$.  We can write $\mathbf{v}$ as a column vector, namely $\mathbf{v}=(v_{1},v_{2},\ldots,v_{n})^{\mathsf{T}}$, where ``$\mathsf{T}$'' denotes transposition of vectors.  Using the matrix representation of $\gamma$, see \eqref{eq:gddn_coxeter_matrix}, we obtain
	\begin{displaymath}
		\mathbf{v}' = \gamma\mathbf{v} = \bigl(\zeta_{d}v_{n-1},v_{1},v_{2},\ldots,v_{n-2},\zeta_{d}^{d-1}v_{n}\bigr)^{\mathsf{T}}.
	\end{displaymath}
	Recall that $w^{-1}=\colref{1}{a}{s}\colref{1}{a}{0}$, since the reflections in $G(d,d,n)$ are involutions.  
	
	If $a<n$, then we have 
	\begin{displaymath}
		\mathbf{v}'' = \colref{1}{a}{0}\mathbf{v}' = \bigl(v_{a-1},v_{1},v_{2},\ldots,v_{a-2},\zeta_{d}v_{n-1},v_{a},\ldots,v_{n-2},\zeta_{d}^{d-1}v_{n}\bigr)^{\mathsf{T}},
	\end{displaymath}
	and thus
	\begin{displaymath}
		\colref{1}{a}{s}\mathbf{v}'' = \bigl(\zeta_{d}^{1-s}v_{n-1},v_{1},v_{2},\ldots,v_{a-2},\zeta_{d}^{s}v_{a-1},v_{a},\ldots,v_{n-2},\zeta_{d}^{d-1}v_{n}\bigr)^{\mathsf{T}}.
	\end{displaymath}
	It follows that $\text{Fix}(w^{-1}\gamma)$ is determined by the following system of linear equations.
	\begin{displaymath}\begin{aligned}
		& v_{1} = \zeta_{d}^{1-s}v_{n-1}, && v_{2} = v_{1}, && v_{3} = v_{2}, && \ldots, && v_{a-1} = v_{a-2}, &&\\
		& v_{a} = \zeta_{d}^{s}v_{a-1}, && v_{a+1} = v_{a}, && a_{a+2} = v_{a+1}, && \ldots, && v_{n-1} = v_{n-2}, && v_{n} = \zeta_{d}^{d-1}v_{n}.
	\end{aligned}\end{displaymath}
	If we put these equations together, we obtain
	\begin{align*}
		\zeta_{d}^{1-s}v_{n-1} & = v_{1} = v_{2} = \cdots = v_{a-1} = \zeta_{d}^{d-s}v_{a} = \zeta_{d}^{d-s}v_{a+1} = \cdots = \zeta_{d}^{d-s}v_{n-1},\\
		\zeta_{d}^{d-1}v_{n} & = v_{n}.
	\end{align*}
	Since $d>1$, it follows that there do not exist nontrivial solutions for these equalities to hold, and hence $\dim\text{Fix}(w^{-1}\gamma)=0$, which implies that in this case $w\not\leq_{T}\gamma$.  

	If $a=n$, then we have	
	\begin{displaymath}
		\mathbf{v}'' = \colref{1}{n}{0}\mathbf{v}' = \bigl(\zeta_{d}^{d-1}v_{n},v_{1},v_{2},\ldots,v_{n-2},\zeta_{d}v_{n-1}\bigr)^{\mathsf{T}},
	\end{displaymath}
	and thus 
	\begin{displaymath}
		\colref{1}{n}{s}\mathbf{v}'' = \bigl(\zeta_{d}^{1-s}v_{n-1},v_{1},v_{2},\ldots,v_{n-2},\zeta_{d}^{s-1}v_{n}\bigr)^{\mathsf{T}},
	\end{displaymath}
	As before, we obtain 
	\begin{align*}
		\zeta_{d}^{1-s}v_{n-1} & = v_{1} = v_{2} = \cdots = v_{n-1},\\
		\zeta_{d}^{s-1}v_{n} & = v_{n}.
	\end{align*}
	Hence there exist nontrivial solutions for these equalities only if $s=1$, which implies that in this case $w\leq_{T}\gamma$ if and only if $s=1$.  This completes the proof.
\end{proof}

Now we can prove Lemma~\ref{lem:empty_chunks}.

\begin{proof}[Proof of Lemma~\ref{lem:empty_chunks}]
	First of all, let $i\in\{2,3,\ldots,n-1\}$ and $s'\in\{0,1,\ldots,d-1\}$.  By definition, every $x\in G(d,d,n)$ with $x\bigl(\colint{1}{0}\bigr)=\colint{i}{s'}$ satisfies $x=\colref{1}{i}{s'}x'$ for some $x'\in G(d,d,n)$ with $x'\bigl(\colint{1}{0}\bigr)=\colint{1}{0}$.  In view of \eqref{eq:absolute_order}, we conclude that $\colref{1}{i}{s'}\leq_{T}x$.  Hence $\colref{1}{i}{s'}\leq_{T}\gamma$ if and only if $x\in R_{i}^{(s')}$.  Lemma~\ref{lem:gddn_atoms} implies that this is the case only if $s'\in\{0,d-1\}$.  This settles the claim that $R_{i}^{(s')}$ is empty if and only if $1\leq s'<d-1$.  
	
	Now let $s\in\{0,1,\ldots,d-1\}$, and pick $x\in G(d,d,n)$ with $x\bigl(\colint{1}{0}\bigr)=\colint{1}{s}$.  If $s=0$, then we can for instance consider $x=\varepsilon$ to see that $R_{1}^{(0)}$ is not empty.  Hence let $s\neq 0$.  In that case, the cycle decomposition of $x$ contains a balanced cycle $w=\Bigl[\colint{1}{0}\Bigr]_{s}$, and we notice that $w$ itself is not an element of $G(d,d,n)$.  However, for every $1<a\leq n$ we find that 
	\begin{displaymath}
		\Bigl[\colint{1}{0}\Bigr]_{s}\Bigl[\colint{a}{0}\Bigr]_{d-s}=\colref{1}{a}{0}\colref{1}{a}{s}
	\end{displaymath}
	belongs to $G(d,d,n)$, and it is immediate that $\Bigl[\colint{1}{0}\Bigr]_{s}\Bigl[\colint{a}{0}\Bigr]_{d-s}\leq_{T}x$.  Again it follows that $\Bigl[\colint{1}{0}\Bigr]_{s}\Bigl[\colint{a}{0}\Bigr]_{d-s}\leq_{T}\gamma$ if and only if $x\in R_{1}^{(s)}$.  Lemma~\ref{lem:fixed_space_first_lemma} implies that this is only possible if $a=n$ and $s=1$.
\end{proof}

\begin{lemma}\label{lem:fixed_space_second_lemma}
	Let $d,n\geq 2$, and let $\gamma\in G(d,d,n)$ be the Coxeter element defined in \eqref{eq:gddn_coxeter_element}.  If $w$ is a coatom of $\pnc{G(d,d,n)}{1}(\gamma)$, and $\colref{1}{2}{d-1}\leq_{T}w$, then we have either
	\begin{displaymath}
		w = \Bigl[\colint{1}{0}\;\ldots\;\colint{a}{0}\;\colint{(b+1)}{0}\;\ldots\;\colint{(n-1)}{0}\Bigr]\Bigl[\colint{n}{0}\Bigr]_{d-1}\Bigll\colint{(a+1)}{0}\;\ldots\;\colint{b}{0}\Bigrr
	\end{displaymath}
	for $1<a<b<n$, or
	\begin{displaymath}
		w = \Bigll\colint{1}{0}\;\colint{n}{s}\;\colint{2}{d-1}\;\ldots\;\colint{(n-1)}{d-1}\Bigrr
	\end{displaymath}
	for $0\leq s<d$.  
\end{lemma}
\begin{proof}
	Suppose that $w$ is a coatom of $\pnc{G(d,d,n)}{1}(\gamma)$.  Then $\ell_{T}(w)=n-1$, and proving that $\colref{1}{2}{d-1}\leq_{T}w$ amounts to showing that $\dim\text{Fix}\Bigl(\colref{1}{2}{d-1}w\Bigr)=2$.  Let $V=\mathbb{C}^{n}$ and pick $\mathbf{v}\in V$, which we again write as a column vector $\mathbf{v}=(v_{1},v_{2},\ldots,v_{n})^{\mathsf{T}}$.  In view of Lemma~\ref{lem:gddn_coatoms} there are three choices for $w$.
	
	(i) Let $w=\Bigl[\colint{1}{0}\;\ldots\;\colint{a}{0}\;\colint{(b+1)}{0}\;\ldots\;\colint{(n-1)}{0}\Bigr]\Bigl[\colint{n}{0}\Bigr]_{d-1}\Bigll\colint{(a+1)}{0}\;\ldots\;\colint{b}{0}\Bigrr$.  We thus have 
	\begin{displaymath}
		\mathbf{v}' = w\mathbf{v} = \bigl(\zeta_{d}v_{n-1},v_{1},v_{2},\ldots,v_{a-1},\underindex{a+1}{v_{b}},v_{a+1},\ldots,v_{b-1},\underindex{b+1}{v_{a}},v_{b+1},\ldots,v_{n-2},\zeta_{d}^{d-1}v_{n}\bigr)^{\mathsf{T}},
	\end{displaymath}
	and consequently
	\begin{multline*}
		\colref{1}{2}{d-1}\mathbf{v}' = \bigl(\zeta_{d}v_{1},v_{n-1},v_{2}\ldots,v_{a-1},\underindex{a+1}{v_{b}},v_{a+1},\ldots,v_{b-1},\underindex{b+1}{v_{a}},v_{b+1},\ldots,v_{n-2},\zeta_{d}^{d-1}v_{n}\bigr)^{\mathsf{T}}.
	\end{multline*}
	Hence $\text{Fix}\Bigl(\colref{1}{2}{d-1}w\Bigr)$ is determined by the equations
	\begin{align*}
		\zeta_{d}v_{1} & = v_{1},\\
		v_{n-1} & = v_{2} = v_{3} = \cdots = v_{a} = v_{b+1} = v_{b+2} = \cdots = v_{n-1},\\
		v_{b} & = v_{a+1} = \cdots = v_{b},\\
		\zeta_{d}^{d-1}v_{n} & = v_{n},
	\end{align*}
	which implies that it has dimension $2$ as desired.
	
	(ii) Let $w=\Bigll(\colint{1}{0}\;\ldots\;\colint{a}{0}\;\colint{(b+1)}{d-1}\;\ldots\;\colint{(n-1)}{d-1}\Bigrr\Bigl[\colint{(a+1)}{0}\;\ldots\;\colint{b}{0}\Bigr]\Bigl[\colint{n}{0}\Bigr]_{d-1}$.  We thus have
	\begin{displaymath}
		\mathbf{v}' = w\mathbf{v} = \bigl(\zeta_{d}v_{n-1},v_{1},v_{2},\ldots,v_{a-1},\underindex{a+1}{\zeta_{d}v_{b}},v_{a+1},\ldots,v_{b-1},\underindex{b+1}{\zeta_{d}^{d-1}v_{a}},v_{b+1},\ldots,v_{n-2},\zeta_{d}^{d-1}v_{n}\bigr)^{\mathsf{T}},
	\end{displaymath}
	and consequently
	\begin{displaymath}
		\colref{1}{2}{d-1}\mathbf{v}' = \bigl(\zeta_{d}v_{1},v_{n-1},v_{2},\ldots,v_{a-1},\underindex{a+1}{\zeta_{d}v_{b}},v_{a+1},\ldots,v_{b-1},\underindex{b+1}{\zeta_{d}^{d-1}v_{a}},v_{b+1},\ldots,v_{n-2},\zeta_{d}^{d-1}v_{n}\bigr)^{\mathsf{T}}.
	\end{displaymath}
	Hence $\text{Fix}\Bigl(\colref{1}{2}{d-1}w\Bigr)$ is determined by the equations
	\begin{align*}
		\zeta_{d}v_{1} & = v_{1},\\
		v_{n-1} & = v_{2} = \cdots = v_{a} = \zeta_{d}v_{b+1} = \cdots = \zeta_{d}v_{n-1},\\
		\zeta_{d}v_{b} & = v_{a+1} = \cdots = v_{b},\\
		\zeta_{d}^{d-1}v_{n} & = v_{n},
	\end{align*}
	which, since $d\geq 2$, implies that it has dimension $0$.  Hence $\colref{1}{2}{d-1}\not\leq_{T}w$ as desired.
	
	(iii) Let $w=\Bigll\colint{1}{0}\;\ldots\;\colint{a}{0}\;\colint{n}{s-1}\;\colint{(a+1)}{d-1}\;\ldots\;\colint{(n-1)}{d-1}\Bigrr$ for $0\leq s<d$.  We thus have
	\begin{displaymath}
		\mathbf{v}' = w\mathbf{v} = \bigl(\zeta_{d}v_{n-1},v_{1},v_{2},\ldots,v_{a-1},\underindex{a+1}{\zeta_{d}^{d-s}v_{n}},v_{a+1},\ldots,v_{n-2},\zeta_{d}^{s-1}v_{a}\bigr)^{\mathsf{T}}.
	\end{displaymath}
	If $a>1$, then we obtain
	\begin{displaymath}
		\colref{1}{2}{d-1}\mathbf{v}' = \bigl(\zeta_{d}v_{1},v_{n-1},v_{2}\ldots,v_{a-1},\underindex{a+1}{\zeta_{d}^{d-s}v_{n}},v_{a+1},\ldots,v_{n-2},\zeta_{d}^{s-1}v_{a}\bigr)^{\mathsf{T}}.
	\end{displaymath}
	Hence $\text{Fix}\Bigl(\colref{1}{2}{d-1}w\Bigr)$ is determined by the equations
	\begin{align*}
		\zeta_{d}v_{1} & = v_{1},\\
		v_{n-1} & = v_{2} = \cdots = v_{a} = \zeta_{d}^{1-s}v_{n} = \zeta_{d}v_{a+1} = \zeta_{d}v_{a+2} = \cdots = \zeta_{d}v_{n-1},
	\end{align*}
	which, since $d\geq 2$, implies that it has dimension $0$.  Hence $\colref{1}{2}{d-1}\not\leq_{T}w$ in this case.  If $a=1$, however, then we obtain
	\begin{displaymath}
		\colref{1}{2}{d-1}\mathbf{v}' = \bigl(\zeta_{d}^{1-s}v_{n},v_{n-1},v_{2},\ldots,v_{n-2},\zeta_{d}^{s-1}v_{1}\bigr)^{\mathsf{T}}.
	\end{displaymath}
	Hence $\text{Fix}\Bigl(\colref{1}{2}{d-1}w\Bigr)$ is determined by the equations
	\begin{align*}
		\zeta_{d}^{1-s}v_{n} & = v_{1} = \zeta_{d}^{1-s}v_{n},\\
		v_{n-1} & = v_{2} = \cdots = v_{n-1},
	\end{align*}
	which implies that it has dimension $2$ as desired.	
\end{proof}

Now we can prove Lemma~\ref{lem:fifth_chunk}.

\begin{proof}[Proof of Lemma~\ref{lem:fifth_chunk}]
		First consider the set $R_{1}^{(1)}$.  Lemma~\ref{lem:gddn_coatoms} implies that
	\begin{displaymath}
		\Bigl[\colint{1}{0}\Bigr]\Bigl[\colint{n}{0}\Bigr]_{d-1}\Bigll\colint{2}{0}\;\ldots\;\colint{(n-1)}{0}\Bigrr
	\end{displaymath}
	is the greatest element of $\RR_{1}^{(1)}$, and has length $n-1$.  Moreover, Lemma~\ref{lem:gddn_atoms} implies that no atom of $\pnc{G(d,d,n)}{1}(\gamma)$ belongs to $R_{1}^{(1)}$.  It follows from Lemma~\ref{lem:fixed_space_first_lemma} that $\Bigl[\colint{1}{0}\Bigr]\Bigl[\colint{n}{0}\Bigr]_{d-1}$ is the least element of $\RR_{1}^{(1)}$.  Hence $\RR_{1}^{(1)}$ is an interval, and Lemmas~\ref{lem:gddn_single_short_cycles} and \ref{lem:bottom_intervals} imply that $\RR_{1}^{(1)}\cong\pnc{G(1,1,n-2)}{1}$ as desired.
	
	\smallskip	

	Now consider the set $R_{2}^{(d-1)}$.  Lemma~\ref{lem:gddn_atoms} implies that $\colref{1}{2}{d-1}\in R_{2}^{(d-1)}$, and it is immediate that for every $x\in R_{2}^{(d-1)}$ we have $\colref{1}{2}{d-1}\leq_{T}x$.  Hence $\RR_{2}^{(d-1)}$ has a least element.  Lemma~\ref{lem:gddn_coatoms} implies that no coatom of $\pnc{G(d,d,n)}{1}(\gamma)$ belongs to $R_{2}^{(d-1)}$.  Lemma~\ref{lem:fixed_space_second_lemma} tells us that the coatoms above $\colref{1}{2}{d-1}$ are either of the form 
	\begin{displaymath}
		\Bigl[\colint{1}{0}\;\ldots\;\colint{a}{0}\;\colint{(b+1)}{0}\;\ldots\;\colint{(n-1)}{0}\Bigr]\Bigl[\colint{n}{0}\Bigr]_{d-1}\Bigll\colint{(a+1)}{0}\;\ldots\;\colint{b}{0}\Bigrr
	\end{displaymath}
	for $1<a<b<n$, or of the form
	\begin{displaymath}
		\Bigll\colint{1}{0}\;\colint{n}{s}\;\colint{2}{d-1}\;\ldots\;\colint{(n-1)}{d-1}\Bigrr
	\end{displaymath}
	for $0\leq s<d$.  It is immediate that there cannot be an element in $G(d,d,n)$ that sends $\colint{1}{0}$ to $\colint{2}{d-1}$ and is at the same time covered by a coatom of the first form, but there is a unique such element that is covered by a coatom of the second form, namely $\Bigll(\colint{1}{0}\;\colint{2}{d-1}\;\ldots\;\colint{(n-1)}{d-1}\Bigrr$.  Hence $\RR_{2}^{(d-1)}$ has a greatest element of length $n-2$.  Lemmas~\ref{lem:gddn_single_short_cycles} and \ref{lem:bottom_intervals} imply that $\RR_{2}^{(d-1)}$ is isomorphic to $\pnc{G(1,1,n-2)}{1}$ as desired.  
\end{proof}

\bibliography{../../literature}

\begin{bibdiv}
\begin{biblist}

\bib{anderson02combinatorics}{book}{
      author={Anderson, Ian},
       title={{Combinatorics of Finite Sets}},
   publisher={Dover Publications},
     address={Mineola, NY},
        date={2002},
}

\bib{armstrong09generalized}{article}{
      author={Armstrong, Drew},
       title={{Generalized Noncrossing Partitions and Combinatorics of Coxeter
  Groups}},
        date={2009},
     journal={Memoirs of the American Mathematical Society},
      volume={202},
}

\bib{athanasiadis07shellability}{article}{
      author={Athanasiadis, Christos~A.},
      author={Brady, Thomas},
      author={Watt, Colum},
       title={{Shellability of Noncrossing Partition Lattices}},
        date={2007},
     journal={Proceedings of the American Mathematical Society},
      volume={135},
       pages={939\ndash 949},
}

\bib{athanasiadis04noncrossing}{article}{
      author={Athanasiadis, Christos~A.},
      author={Reiner, Victor},
       title={{Noncrossing Partitions for the Group $D_{n}$}},
        date={2004},
     journal={SIAM Journal on Discrete Mathematics},
      volume={18},
       pages={397\ndash 417},
}

\bib{bessis03dual}{article}{
      author={Bessis, David},
       title={{The Dual Braid Monoid}},
        date={2003},
     journal={Annales Scientifiques de l'{\'E}cole Normale Sup{\'e}rieure},
      volume={36},
       pages={647\ndash 683},
}

\bib{bessis15finite}{article}{
      author={Bessis, David},
       title={{Finite Complex Reflection Arrangements are $K(\pi,1)$}},
        date={2015},
     journal={Annals of Mathematics},
      volume={181},
       pages={809\ndash 904},
}

\bib{bessis06non}{article}{
      author={Bessis, David},
      author={Corran, Ruth},
       title={{Non-crossing Partitions of Type $(e,e,r)$}},
        date={2006},
     journal={Advances in Mathematics},
      volume={202},
       pages={1\ndash 49},
}

\bib{biane97some}{article}{
      author={Biane, Philippe},
       title={{Some Properties of Crossings and Partitions}},
        date={1997},
     journal={Discrete Mathematics},
      volume={175},
       pages={41\ndash 53},
}

\bib{bjorner80shellable}{article}{
      author={Bj{\"o}rner, Anders},
       title={{Shellable and Cohen-Macaulay Partially Ordered Sets}},
        date={1980},
     journal={Transactions of the American Mathematical Society},
      volume={260},
       pages={159\ndash 183},
}

\bib{brady01partial}{article}{
      author={Brady, Thomas},
       title={{A Partial Order on the Symmetric Group and new $K(\pi,1)$'s for
  the Braid Groups}},
        date={2001},
     journal={Advances in Mathematics},
      volume={161},
       pages={20\ndash 40},
}

\bib{brady02partial}{article}{
      author={Brady, Thomas},
      author={Watt, Colum},
       title={{A Partial Order on the Orthogonal Group}},
        date={2002},
     journal={Communications in Algebra},
      volume={30},
       pages={3749\ndash 3754},
}

\bib{brady02artin}{article}{
      author={Brady, Thomas},
      author={Watt, Colum},
       title={{$K(\pi,1)$'s for Artin Groups of Finite Type}},
        date={2002},
     journal={Geometriae Dedicata},
      volume={94},
       pages={225\ndash 250},
}

\bib{brady08non}{article}{
      author={Brady, Thomas},
      author={Watt, Colum},
       title={{Non-Crossing Partition Lattices in Finite Real Reflection
  Groups}},
        date={2008},
     journal={Transactions of the American Mathematical Society},
      volume={360},
       pages={1983\ndash 2005},
}

\bib{broue98complex}{article}{
      author={Brou{\'e}, Michel},
      author={Malle, Gunter},
      author={Rouquier, Rapha{\"e}l},
       title={{Complex Reflection Groups, Braid Groups, Hecke Algebras}},
        date={1998},
     journal={Journal f\"ur die reine und angewandte Mathematik},
      volume={500},
       pages={127\ndash 190},
}

\bib{canfield78on}{article}{
      author={Canfield, E.~Rodney},
       title={{On A Problem of Rota}},
        date={1978},
     journal={Advances in Mathematics},
      volume={29},
       pages={1\ndash 10},
}

\bib{chapoton04enumerative}{article}{
      author={Chapoton, Fr{\'e}d{\'e}ric},
       title={{Enumerative Properties of Generalized Associahedra}},
        date={2004},
     journal={S{\'e}minaire Lotharingien de Combinatiore},
      volume={51},
}

\bib{chapuy14counting}{article}{
      author={Chapuy, Guillaume},
      author={Stump, Christian},
       title={{Counting Factorizations of Coxeter Elements into Products of
  Reflections}},
        date={2014},
     journal={Journal of the London Mathematical Society},
      volume={90},
       pages={919\ndash 939},
}

\bib{chevalley55invariants}{article}{
      author={Chevalley, Claude},
       title={{Invariants of Finite Groups Generated by Reflections}},
        date={1955},
     journal={American Journal of Mathematics},
      volume={77},
       pages={778\ndash 782},
}

\bib{davey02introduction}{book}{
      author={Davey, Brian~A.},
      author={Priestley, Hilary~A.},
       title={{Introduction to Lattices and Order}},
   publisher={Cambridge University Press},
     address={Cambridge},
        date={2002},
}

\bib{bruijn51on}{article}{
      author={de~Bruijn, Nicolaas~G.},
      author={van Ebbenhorst~Tengbergen, Cornelia~A.},
      author={Kruyswijk, Dirk},
       title={{On the Set of Divisors of a Number}},
        date={1951},
     journal={Nieuw Archief voor Wiskunde},
      volume={23},
       pages={191\ndash 193},
}

\bib{deligne74letter}{misc}{
      author={Deligne, Pierre},
       title={{Letter to Eduard Looijenga}},
        date={1974},
        note={Available at
  \url{http://homepage.univie.ac.at/christian.stump/Deligne_Looijenga_Letter_09-03-1974.pdf}},
}

\bib{dilworth71counterexample}{article}{
      author={Dilworth, Robert~P.},
      author={Greene, Curtis},
       title={{A Counterexample to the Generalization of Sperner's Theorem}},
        date={1971},
     journal={Journal of Combinatorial Theory},
      volume={10},
       pages={18\ndash 21},
}

\bib{edelman80chain}{article}{
      author={Edelman, Paul~H.},
       title={{Chain Enumeration and Non-Crossing Partitions}},
        date={1980},
     journal={Discrete Mathematics},
      volume={31},
       pages={171\ndash 180},
}

\bib{engel97sperner}{book}{
      author={Engel, Konrad},
       title={{Sperner Theory}},
   publisher={Cambridge University Press},
     address={Cambridge},
        date={1997},
}

\bib{erdos45on}{article}{
      author={Erd{\H o}s, Paul},
       title={{On a Lemma of Littlewood and Offord}},
        date={1945},
     journal={Bulletin of the American Mathematical Society},
      volume={51},
       pages={898\ndash 902},
}

\bib{griggs77sufficient}{article}{
      author={Griggs, Jerrold~R.},
       title={{Sufficient Conditions for a Symmetric Chain Order}},
        date={1977},
     journal={SIAM Journal on Applied Mathematics},
      volume={32},
       pages={807\ndash 809},
}

\bib{griggs80on}{article}{
      author={Griggs, Jerrold~R.},
       title={{On Chains and Sperner $k$-Families in Ranked Posets}},
        date={1980},
     journal={Journal of Combinatorial Theory (Series A)},
      volume={28},
       pages={156\ndash 168},
}

\bib{hersh99deformation}{article}{
      author={Hersh, Patricia},
       title={{Deformation of Chains via a Local Symmetric Group Action}},
        date={1999},
     journal={The Electronic Journal of Combinatorics},
      volume={6},
}

\bib{jichang84superantichains}{article}{
      author={Jichang, Sha},
      author={Kleitman, Daniel~J.},
       title={{Superantichains in the Lattice of Partitions of a Set}},
        date={1984},
     journal={Studies in Applied Mathematics},
      volume={71},
       pages={207\ndash 241},
}

\bib{kreweras72sur}{article}{
      author={Kreweras, Germain},
       title={Sur les partitions non crois\'ees d'un cycle},
        date={1972},
     journal={Discrete Mathematics},
      volume={1},
       pages={333\ndash 350},
}

\bib{leclerc94finite}{article}{
      author={Leclerc, Bruno},
       title={{A Finite Coxeter Group the Weak Bruhat Order of Which is not
  Symmetric Chain}},
        date={1994},
     journal={European Journal of Combinatorics},
      volume={15},
       pages={181\ndash 185},
}

\bib{lehrer99reflection}{article}{
      author={Lehrer, Gustav~I.},
      author={Springer, Tonny~A.},
       title={{Reflection Subquotients of Unitary Reflection Groups}},
        date={1999},
     journal={Canadian Journal of Mathematics},
      volume={51},
       pages={1175\ndash 1193},
}

\bib{lehrer09unitary}{book}{
      author={Lehrer, Gustav~I.},
      author={Taylor, Donald~E.},
       title={{Unitary Reflection Groups}},
   publisher={Cambridge University Press},
     address={Cambridge},
        date={2009},
}

\bib{mccammond06noncrossing}{article}{
      author={McCammond, Jon},
       title={{Noncrossing Partitions in Surprising Locations}},
        date={2006},
     journal={American Mathematical Monthly},
      volume={113},
       pages={598\ndash 610},
}

\bib{michel15development}{article}{
      author={Michel, Jean},
       title={{The Development Version of the CHEVIE Package of GAP3}},
        date={2015},
     journal={Journal of Algebra},
      volume={435},
       pages={308\ndash 336},
}

\bib{muehle15el}{article}{
      author={M{\"u}hle, Henri},
       title={{EL-Shellability and Noncrossing Partitions Associated with
  Well-Generated Complex Reflection Groups}},
        date={2015},
     journal={European Journal of Combinatorics},
      volume={43},
       pages={249\ndash 278},
}

\bib{orlik80unitary}{article}{
      author={Orlik, Peter},
      author={Solomon, Louis},
       title={{Unitary Reflection Groups and Cohomology}},
        date={1980},
     journal={Inventiones Mathematicae},
      volume={59},
       pages={77\ndash 94},
}

\bib{petersen13on}{article}{
      author={Petersen, T.~Kyle},
       title={{On the Shard Intersection Order of a Coxeter Group}},
        date={2013},
     journal={SIAM Journal on Discrete Mathematics},
      volume={27},
       pages={1880\ndash 1912},
}

\bib{petersen15eulerian}{book}{
      author={Petersen, T.~Kyle},
       title={{Eulerian Numbers}},
   publisher={Birkh{\"a}user},
     address={Basel},
        date={2015},
}

\bib{proctor80product}{article}{
      author={Proctor, Robert~A.},
      author={Saks, Michael~E.},
      author={Sturtevant, Dean~G.},
       title={{Product Partial Orders with the Sperner Property}},
        date={1980},
     journal={Discrete Mathematics},
      volume={30},
       pages={173\ndash 180},
}

\bib{reading08chains}{article}{
      author={Reading, Nathan},
       title={{Chains in the Noncrossing Partition Lattice}},
        date={2008},
     journal={SIAM Journal on Discrete Mathematics},
      volume={22},
       pages={875\ndash 886},
}

\bib{reiner97non}{article}{
      author={Reiner, Victor},
       title={{Non-Crossing Partitions for Classical Reflection Groups}},
        date={1997},
     journal={Discrete Mathematics},
      volume={177},
       pages={195\ndash 222},
}

\bib{reiner14on}{article}{
      author={Reiner, Victor},
      author={Ripoll, Vivien},
      author={Stump, Christian},
       title={{On Non-Conjugate Coxeter Elements in Well-Generated Reflection
  Groups}},
        date={2014},
      eprint={arXiv:1404.5522},
}

\bib{ripoll10orbites}{article}{
      author={Ripoll, Vivien},
       title={Orbites d'{H}urwitz des factorisations primitives d'un
  {\'e}l{\'e}ment de {C}oxeter},
        date={2010},
     journal={Journal of Algebra},
      volume={323},
       pages={1432\ndash 1453},
}

\bib{shearer79simple}{article}{
      author={Shearer, James~B.},
       title={{A Simple Counterexample to a Conjecture of Rota}},
        date={1979},
     journal={Discrete Mathematics},
      volume={28},
       pages={327\ndash 330},
}

\bib{shephard54finite}{article}{
      author={Shephard, Geoffrey~C.},
      author={Todd, John~A.},
       title={{Finite Unitary Reflection Groups}},
        date={1954},
     journal={Canadian Journal of Mathematics},
      volume={6},
       pages={274\ndash 304},
}

\bib{simion00noncrossing}{article}{
      author={Simion, Rodica},
       title={{Noncrossing Partitions}},
        date={2000},
     journal={Discrete Mathematics},
      volume={217},
       pages={397\ndash 409},
}

\bib{simion91on}{article}{
      author={Simion, Rodica},
      author={Ullman, Daniel},
       title={{On the Structure of the Lattice of Noncrossing Partitions}},
        date={1991},
     journal={Discrete Mathematics},
      volume={98},
       pages={193\ndash 206},
}

\bib{sperner28ein}{article}{
      author={Sperner, Emanuel},
       title={{Ein Satz {\"u}ber Untermengen einer endlichen Menge}},
        date={1928},
     journal={Mathematische Zeitschrift},
      volume={27},
       pages={544\ndash 548},
}

\bib{springer74regular}{article}{
      author={Springer, Tonny~A.},
       title={{Regular Elements of Finite Reflection Groups}},
        date={1974},
     journal={Inventiones Mathematicae},
      volume={25},
       pages={159\ndash 198},
}

\bib{stanley80weyl}{article}{
      author={Stanley, Richard~P.},
       title={{Weyl Groups, the Hard Lefschetz Theorem, and the Sperner
  Property}},
        date={1980},
     journal={SIAM Journal on Algebraic Discrete Methods},
      volume={1},
       pages={168\ndash 184},
}

\bib{sagecombinat}{misc}{
      author={{The Sage-Combinat Community}},
       title={{\texttt{sage-combinat}: Enhancing Sage as a Toolbox for Computer
  Exploration in Algebraic Combinatorics}},
        date={2016},
        note={\url{http://combinat.sagemath.org}},
}

\bib{sage}{misc}{
      author={{The Sage Developers}},
       title={{Sage Mathematics Software System (Version 7.2)}},
        date={2016},
        note={\url{http://www.sagemath.org}},
}

\end{biblist}
\end{bibdiv}

\end{document}